\documentclass[12pt]{amsart}
\usepackage{amsthm,amsmath,amssymb,amscd,graphicx,enumerate, stmaryrd,xspace,verbatim, epic, eepic,color,url}

\usepackage{eurosym}
\usepackage[all]{xypic}
\SelectTips{cm}{}
\usepackage{pdfsync}

{
   \newtheorem{theorem}[subsubsection]{Theorem}
      \newtheorem*{theorem*}{Theorem}
   \newtheorem{proposition}[subsubsection]{Proposition} 
   
   \newtheorem{prop}[subsubsection]{Proposition}     
   \newtheorem{lemma}[subsubsection]{Lemma}

   \newtheorem{corollary}[subsubsection]{Corollary}

   \newtheorem*{conjecture*}{Conjecture}
   
}
{\theoremstyle{definition}
          \newtheorem*{exercise*}{Exercise}
   
   \newtheorem{example}[subsubsection]{Example}
   \newtheorem*{example*}{Example}
   \newtheorem{definition}[subsubsection]{Definition}
   
   \newtheorem*{definition*}{Definition}
   
   \newtheorem{remark}[subsubsection]{Remark}

\newcommand{\RR}{{\mathbb{R}}}

\newcommand{\QQ}{{\mathbb{Q}}}
\newcommand{\NN}{{\mathbb{N}}}

\newcommand{\PP}{{\mathbb{P}}}
\newcommand{\ZZ}{{\mathbb{Z}}}

\newcommand{\GG}{{\mathbb{G}}}

\renewcommand{\AA}{{\mathbb{A}}}

\newcommand{\bG}{{\mathbf{G}}}

\newcommand{\bGamma}{{\boldsymbol{\Gamma}}}

\newcommand{\cC}{{\mathcal C}}

\newcommand{\cM}{{\mathcal M}}

\newcommand{\cO}{{\mathcal O}}

\newcommand{\cX}{{\mathcal X}}
\newcommand{\cY}{{\mathcal Y}}
\newcommand{\cZ}{{\mathcal Z}}

\newcommand{\im}{{\operatorname{Im}}}

\def\<{\langle}
\def\>{\rangle}

\newcommand{\Spec}{\operatorname{Spec}}

\newcommand{\Hom}{{\operatorname{Hom}}}

\newcommand{\Aut}{{\operatorname{Aut}}}

\newcommand{\ocM}{\overline{{\mathcal M}}}

\newcommand{\oM}{{\overline{M}}}

\newcommand{\oSigma}{{\overline{\Sigma}}}
\newcommand{\osigma}{{\overline{\sigma}}}
\newcommand{\otau}{{\overline{\tau}}}
\newcommand{\oPhi}{{\overline\Phi}}

\newcommand{\Trop}{{\operatorname{Trop}}}
\newcommand{\trop}{{\operatorname{trop}}}

\newcommand{\An}{{\operatorname{an}}}
\newcommand{\an}{{\operatorname{an}}}
\newcommand{\val}{{\operatorname{val}}}

\newcommand{\TC}{\bGamma}
\newcommand{\WG}{{\mathbf G}}

\newcommand{\boldp}{{\mathbf p}}

\newcommand{\double}{\genfrac..{0pt}1
{\raise -2pt\hbox{$\scriptstyle\longrightarrow$}}{\raise 4pt\hbox
{$\scriptstyle\longrightarrow$}}} 

\renewcommand{\setminus}{\smallsetminus}

\def\tototi{\mathbin{\mathop{\otimes}\limits^{\raise-1pt\hbox
{$\scriptscriptstyle {\rm L}$}}}}

\def\indlim{\mathop{\vrule width0pt height7pt depth
4pt\smash{\lim\limits_{\raise 1pt\hbox to 14.5pt
{\rightarrowfill}}}}}
\def\projlim{\mathop{\vrule width0pt height7pt depth
4pt\smash{\lim\limits_{\raise 1pt\hbox to 14.5pt
{\leftarrowfill}}}}}

\newcommand\displaceamount{3pt}

\newcommand{\doubledown}{\ar@<\displaceamount>[d]\ar@<-\displaceamount>[d]}

\newcommand{\doubleup}{\ar@<\displaceamount>[u]\ar@<-\displaceamount>[u]}

\newcommand{\doubleright}{\ar@<\displaceamount>[r]\ar@<-\displaceamount>[r]}

\def\sp{q}

\begin{document}

\title{The tropicalization of the moduli space of curves}

\author[Abramovich]{Dan Abramovich}

\author[Caporaso]{Lucia Caporaso}

\author[Payne]{Sam Payne}

\address[Abramovich]{Department of Mathematics\\
Brown University\\
Box 1917\\
Providence, RI 02912\\
U.S.A.}
\email{abrmovic@math.brown.edu}

\address[Caporaso]{Dipartimento di Matematica e Fisica\\
Universit\`{a} Roma Tre\\
Largo San Leonardo Murialdo \\
I-00146 Roma\\  Italy }
\email{caporaso@mat.uniroma3.it}

\address[Payne]{Yale University
Mathematics Department\\
10 Hillhouse Ave\\
New Haven, CT 06511 \\ U.S.A.}
\email{sam.payne@yale.edu}

\thanks{Abramovich supported in part by NSF grants DMS-0901278, DMS-1162367 and a Lady Davis fellowship.  Payne supported in part by NSF grant DMS-1068689 and NSF CAREER grant DMS-1149054.}
\date{\today}

\begin{abstract}
We show that the skeleton of the Deligne-Mumford-Knudsen
moduli stack of stable curves is naturally identified with the moduli
space of extended tropical curves, and that this is compatible with the ``naive" set-theoretic tropicalization map. The proof passes through general structure results on the skeleton of a toroidal Deligne-Mumford stack. Furthermore, we construct tautological forgetful, clutching, and gluing maps between moduli spaces of extended tropical curves and show that they are compatible with the analogous tautological maps in the algebraic setting.
 \end{abstract}

\dedicatory{
Dedicated to Joe Harris.}

\maketitle

\setcounter{tocdepth}{1}
\tableofcontents

\section[test]{Introduction}
A number of researchers have introduced and studied the moduli spaces $M_{g,n}^{\trop}$, parametrizing certain metric weighted graphs called \emph{tropical curves},  and exhibited analogies to the Deligne-Mumford-Knudsen moduli stacks of stable pointed curves, $\ocM_{g,n}$, and to the Kontsevich moduli spaces of stable maps \cite{Mikhalkin1,Mikhalkin2,Gathmann-Markwig,Gathmann-Kerber-Markwig,Kozlov,CV,Kozlov2,BMV,C2,Chan, Chan-Melo-Viviani}.   The paper \cite{C2} describes, in particular, an order reversing correspondence between the stratification of  $M_{g,n}^{\trop}$ and the stratification of     $\ocM_{g,n}$, along with a natural compactification  $\oM_{g,n}^{\trop}$, the moduli space of {\em extended} tropical curves,  where the correspondence persists.  A seminal precursor for all of this work is the paper of Culler and Vogtmann on moduli of graphs and automorphisms of free groups, \cite{Culler-Vogtmann}, in which a space of metric graphs  called  ``outer space" was introduced.

The analogies between moduli of curves and moduli of graphs go further than the natural stratifications of compactifications.  As we show in Section \ref{Sec:tautological}, the moduli spaces  $\oM_{g,n}^{\trop}$ admit natural  maps $$\pi_{g,n}^\trop:\oM_{g,n+1}^{\trop}\to \oM_{g,n}^{\trop},\ \  i=1,\ldots,n+1$$ associated to ``forgetting the last marked point and stabilizing," analogous to the forgetful maps $\pi_{g,n}$ on the moduli spaces of curves. There are also clutching and gluing maps $$\kappa_{g_1,n_1,g_2,n_2}^\trop: \oM_{g_1,n_1+1}^{\trop} \times \oM_{g_2,n_2+1}^{\trop} \to  \oM_{g_1+g_2,n_1+n_2}^{\trop}$$  and $$\gamma_{g,n}^\trop: \oM_{g-1,n+2}^{\trop} \to \oM_{g,n}^{\trop}$$ covering the boundary strata of $\oM_{g,n}^{\trop} \setminus M_{g,n}^{\trop}$, analogous to the corresponding clutching and gluing maps $\kappa_{g_1,n_1,g_2,n_2}$ and $\gamma_{g,n}$ on the moduli spaces of curves.  When the various subscripts $g,n$ are evident we suppress them in the notation   for these maps.

The main purpose of this paper is to develop these analogies into a rigorous and functorial correspondence. Write $\oM_{g,n}$ for the coarse moduli space of $\ocM_{g,n}$. We start with set-theoretic maps from the associated Berkovich analytic space $\oM_{g,n}^\an$ to the tropical moduli space $\oM_{g,n}^\trop$, described in Definition \ref{Def:Trop} below, and use Thuillier's construction of canonical skeletons of toroidal Berkovich spaces  \cite{Thuillier} to show that these maps are continuous, proper, surjective, and compatible with the tautological forgetful, clutching, and gluing maps. We work extensively with the combinatorial geometry of {\em extended generalized cone complexes}, as presented in Section \ref{Sec:generalized}.

To study the skeleton of $\ocM_{g,n}$, we require a mild generalization of Thuillier's construction, presented in Section \ref{Sec:toroidalDM}, below; the main technical results are Propositions \ref{prop:Limit}, \ref{Prop:functoriality} and \ref{Prop:decomposition}. Given a proper toroidal Deligne--Mumford stack $\cX$ with coarse moduli space $X$, we  functorially construct an extended generalized cone complex, the \emph{skeleton} $\oSigma(\cX)$, which is both a topological closed subspace of the Berkovich analytic space $X^\an$ associated to $X$, and also the image of a canonical retraction 
\[
\boldp_\cX:X^\An\to \oSigma(\cX).
\]
We emphasize that the skeleton $\oSigma(\cX)$ depends on a toroidal structure on the stack $\cX$, but lives in the analytification of the coarse moduli space $X$, which is not necessarily toroidal.

The compactified moduli space of tropical curves $\oM_{g,n}^\trop$ is similarly an extended generalized cone complex,  and one of our primary tasks is to identify the tropical moduli space $\oM_{g,n}^\trop$ with the skeleton $\oSigma(\ocM_{g,n})$.
   See Theorem~\ref{Th:functor} for a precise statement.

\subsection{The tropicalization map}
There is a natural set theoretic {\em tropicalization map} $$\Trop: \oM_{g,n}^\an \to \oM_{g,n}^{\trop},$$ well-known to experts \cite{Tyomkin, BPR,  Viviani}, defined as follows.  A point $[C]$ in $\oM_{g,n}^\An$ is represented, possibly after a field extension, by a stable $n$-pointed curve $C$ of genus $g$ over the spectrum $S$ of a valuation ring $R$, with algebraically closed fraction field and valuation denoted $\val_C$. Let $\WG$ be the dual graph of the special fiber, as discussed in Section~\ref{Sec:dualgraph} below, where each vertex is weighted by the genus of the corresponding irreducible component, and with legs corresponding to the marked points.  For each edge $e_i$ in $\WG$, choose an \'etale neighborhood of the corresponding node in which the curve is defined by a local equation $xy = f_i$, with $f_i$ in $R$.  

\begin{definition}\label{Def:Trop}
The tropicalization of the point $[C] \in \oM_{g,n}^\An$ is the stable tropical curve $\TC = (\WG, \ell)$, with edge lengths given by
\[
\ell(e_i) = \val_C(f_i).
\]

\end{definition}

 \noindent See \cite[Lemma 2.2.4]{Viviani} for a proof that  the tropical curve $\TC$ so defined is independent of the choices of $R$, $C$, \'etale neigborhood, and local defining equation, so the map $\Trop$ is well defined.

\subsection{Main results}
Our first main result identifies the map $\Trop$ with the projection from $\oM_{g,n}^\An$ to its skeleton $\oSigma(\ocM_{g,n})$.

\begin{theorem}\label{Th:functor}
Let $g$ and $n$ be non-negative integers.
\begin{enumerate}
\item 
There is an isomorphism of generalized cone complexes with integral structure
$$\Phi_{g,n}: \Sigma(\ocM_{g,n}) \overset\sim\longrightarrow  M_{g,n}^{\trop}$$
extending uniquely to the compactifications
$$\oPhi_{g,n}:\oSigma(\ocM_{g,n}) \overset\sim\longrightarrow  \oM_{g,n}^{\trop}.$$
\item The following diagram is commutative:
$$\xymatrix{ \oM_{g,n}^\an \ar[rr]^{\boldp_{\ocM_{g,n}}}\ar[drr]_\Trop && \oSigma(\ocM_{g,n}) \ar[d]^{\oPhi_{g,n}} \\ 
&& \oM_{g,n}^{\trop}.
}$$
 In particular the map $\Trop$ is continuous, proper, and surjective.
\end{enumerate}
\end{theorem}

The theorem is proven in Section \ref{Sec:proofs}. 

Our second main result shows that the map $\Trop$ is compatible with the tautological forgetful, clutching, and gluing maps.

\begin{theorem}\label{Th:tautological}
The following diagrams are commutative.

The universal curve diagram:
$$\xymatrix{
\oM_{g,n+1}^\an\ar[rr]^{\Trop}\ar[d]_{\pi^\an} &&
 \oM_{g,n+1}^{\trop}\ar[d]^{\pi^\trop} \\
\oM_{g,n}^\an\ar[rr]^{\Trop} && \oM_{g,n}^{\trop}
,} $$ 
the gluing diagram:  $$ \xymatrix{
\oM_{g-1,n+2}^\an\ar[rr]^{\Trop}\ar[d]_{\gamma^\an} &&
 \oM_{g-1,n+2}^{\trop}\ar[d]^{\gamma^\trop} \\
\oM_{g,n}^\an\ar[rr]^{\Trop} && 
\oM_{g,n}^{\trop},
}
$$
and the clutching diagram:
$$\xymatrix{
\oM_{g_1,n_1+1}^\an\times \oM_{g_2,n_2+1}^\an \ar[rrr]^{\Trop\times \Trop}\ar[d]_{\kappa^\an} &&&
 \oM_{g_1,n_1+1}^{\trop} \times \oM_{g_2,n_2+1}^{\trop}\ar[d]^{\kappa^\trop} \\
\oM_{g,n}^\an\ar[rrr]^{\Trop} && &
\oM_{g,n}^{\trop}.
}$$
\end{theorem}

Both notation and proofs are provided in Section \ref{Sec:tautological}

\subsection{Fans, complexes, skeletons and tropicalization} \label{Sec:tropicalization}
There are several combinatorial constructions in the literature relating algebraic varieties to polyhedral cone complexes, and we move somewhat freely among them in this paper.  
 The following is a brief description of the key basic notions,
more details will be given in the sequel. 

Classical tropicalization studies a subvariety of a torus $T$ over a valued field by looking at its image in $N_\RR$, the real extension of the lattice of 1-parameter subgroups of the torus, under the coordinate-wise valuation map.  This basic idea has been generalized in several ways.  For algebraic subvarieties of toric varieties, there are extended tropicalization maps to natural partial compactifications on $N_\RR$ \cite{Kajiwara, Payne, Rabinoff}.  Similar ideas about extending and compactifying tropicalizations appeared earlier in \cite{Mikhalkin2, SS}.

Tropicalization is closely related to several other classical constructions:

\subsubsection{Fans of toric varieties} A toric variety $X$ with dense torus $T$ corresponds naturally to a fan $\Sigma(X)$ in $N_\RR$.  These appear in \cite[I.2]{KKMS}, where they are called ``f.r.p.p. decompositions''.  See also \cite[1.1]{Oda} and \cite[1.4]{FultonToric}.  One key feature of fans, as opposed to abstract cone complexes, is that all of the cones in a fan come with a fixed embedding in an ambient vector space.

\subsubsection{Complexes of toroidal embeddings} In \cite[Chapter II]{KKMS}, the construction associating a fan to a toric variety is generalized to spaces that look locally sufficiently like toric varieties.  To each toroidal embedding without self-intersection $U \subset X$, they associate an abstract rational polyhedral cone complex with integral structure, also denoted $\Sigma(X)$.  Some authors also refer to these cone complexes as fans \cite{KatoToric, Thuillier}, although they do not come with an embedding in an ambient vector space.  For a toroidal embedding $U \subset X$ with self-intersections, Thuillier constructs a generalized cone complex, obtained as a colimit of a finite diagram of rational polyhedral cones with integral structure, which we again denote $\Sigma(X)$.  See \cite[3.3.2]{Thuillier}.  Note that both fans and cone complexes associated to toroidal embeddings without self-intersection are special cases of Thuillier's construction, so the notation is not ambiguous.

\subsubsection{Extended complexes and skeletons}\label{Sec:skeletons} Thuillier also introduced natural compactifications of his generalized cone complexes; the more classical $\Sigma(X)$ is an open dense subset of this extended generalized cone complex $\oSigma(X)$.  The boundary $\oSigma(X) \smallsetminus \Sigma(X)$, is sometimes called the ``part at infinity", and then $\Sigma(X)$ is referred to as the ``finite part" of $\oSigma(X)$.  See  \cite[Sections~3.1.2 and 3.3.2]{Thuillier}.  The extended generalized cone complex $\oSigma(X)$ is also called the \emph{skeleton} of the toroidal scheme $X$.  It is an instance of a skeleton of a Berkovich space \cite{Berkovich-contraction,Hrushovski-Loeser}, and comes with a canonical retraction
\[
\boldp: X^\beth \rightarrow \oSigma(X)
\]
such that $\boldp^{-1}(\Sigma(X)) = X^\beth \cap U^\An$.  Here we use the notation $X^\an$ for the usual Berkovich analytic space of $X$, and  $X^\beth$ is the subset of $X^\an$ consisting of points over valued fields that extend to $\Spec$ of the valuation ring.  Such an extensions is unique when it exists, since varieties are separated.
Notice that if $X$ is proper  
$X^\An=X^\beth$, hence
$\boldp$ is a canonical retraction of $X^\An$ onto the skeleton $\oSigma(X)$ that maps $U^\An$ onto the cone complex $\Sigma(X)$. 
In Section~\ref{Sec:toroidalDM}, we extend these constructions to toroidal embeddings of Deligne--Mumford stacks.  This generalization is straightforward; no new ideas are needed.

\subsubsection{Logarithmic geometry.} The cone complex of \cite{KKMS} is reinterpreted in terms of monoids and ``fans" in the logarithmic setting in \cite[Section~9]{Kato}.  See also \cite[Appendix B]{Gross-Siebert}, where the complex associated to a logarithmic scheme $X$ is called the {\em tropicalization} of $X$.  

\subsubsection{Tropicalization.} Roughly speaking all of these fans, cone complexes, and skeletons are in some sense tropicalizations of the corresponding varieties.  Put another way, tropical geometry may be interpreted as the study of skeletons of Berkovich analytifications.  The exact relation between compactifications of subvarieties of tori and classical tropicalization is explained by the theory of \emph{geometric tropicalization}, due to Hacking, Keel, and Tevelev \cite{Tevelev07, HackingKeelTevelev09}.

We revisit the relations between tropicalization and skeletons of toroidal embeddings in more detail in Sections \ref{Sec:Thuillier} and \ref{Sec:toroidalDM}. 

\subsubsection{Recent progress} The tropicalization of toroidal structures discussed here was extended to logarithmic schemes in \cite{Ulirsch}. The paper \cite{CavalieriMarwigRanganathan} studies the tropicalization of Hurwitz schemes. Spaces of stable maps in a broad non-archimedean context, and their maps to spaces of tropical maps, are studied in \cite{Yu}. 

\subsection{Acknowledgements}
Thanks are due to 
M. Chan,
J. Denef,
J. Rabinoff,
T. Schlanck,
M. Temkin,
I. Tyomkin,
M. Ulirsch,
A. Vistoli,
F. Viviani, and 
J. Wise
for helpful conversations on the subject of this article. We thank the referees for a careful reading and many insightful comments.  The article was initiated during AGNES at MIT, Spring 2011; we thank J. McKernan and C. Xu for creating that opportunity.

\section{Extended and generalized cone complexes}

\label{Sec:complexes}

\subsection{Cones} 
A \emph{polyhedral cone with integral structure}
 $(\sigma,M)$ is a topological space $\sigma$, together with a finitely generated abelian group $M$ of continuous real-valued functions on $\sigma$, such that the induced map $\sigma \rightarrow \Hom(M, \RR)$ is a homeomorphism onto a strictly convex polyhedral cone in the real vector space dual to $M$.  The cone is \emph{rational} if its image is rational with respect to the dual lattice $\Hom(M, \ZZ)$.  A morphism of polyhedral cones with integral structure $(\sigma, M) \rightarrow (\sigma', M')$ is a continuous map from $\sigma$ to $\sigma'$ such that the pullback of any function in $M'$ is in $M$.  

 Throughout, all of the cones that we consider are rational polyhedral cones with integral structure, and we refer to them simply as \emph{cones}.  When no confusion seems possible, we write just $\sigma$ for the cone $(\sigma, M)$.

Let $\sigma$ be a cone, and let $S_\sigma$ be the monoid of linear functions $u \in M$ that are nonnegative on $\sigma$.  Then $\sigma$ is canonically identified with the space of monoid homomorphisms
\[
\sigma = \Hom(S_\sigma, \RR_{\geq 0}),
\]
where $\RR_{\geq 0}$ is taken with its additive monoid structure.  A \emph{face} of $\sigma$ is the subset $\tau$ where some linear function $u \in S_\sigma$ vanishes.  Each face inherits an integral structure, by restricting the functions in $M$. 

The category of cones with cone morphisms does not contain colimits; if we glue a cone to a cone along a morphism, the result is not necessarily a cone.  We now discuss cone complexes, an enlargement of the category of cones with cone morphisms in which one can glue cones along faces.

\subsection{Cone complexes} A {\em rational cone complex with integral structure} is a topological space together with a finite collection of closed subspaces and an identification of each of these closed subspaces with a rational cone with integral structure such that the intersection of any two cones is a union of faces of each.  In other words, it is a topological space presented as the colimit of a poset in the category of cones in which all arrows are isomorphisms onto proper faces.  See \cite[II.1]{KKMS} and \cite[Section~2]{tvbs} for further details.  All of the cones that we consider are rational with integral structure, so we simply call these spaces \emph{cone complexes}.  We refer to the faces of the cones in a complex as the \emph{faces} of the complex.

A morphism of cone complexes $f: \Sigma \rightarrow \Sigma'$ is a continuous map of topological spaces such that, for each cone $\sigma$ in $\Sigma$ there is a cone $\sigma'$ in $\Sigma'$ such that $f|_\sigma$ factors through a morphism of cones $\sigma \rightarrow \sigma'$.

\begin{remark} \label{rem:colimits}
Although we have described cone complexes topologically, one could give an equivalent description in categorical language, as follows.  Let $D$ be a poset in the category of cones in which each arrow is an isomorphism onto a proper face in the target.  The cone complex obtained by gluing the cones in $D$ along these proper face morphisms is the set-valued functor on the category of cones $\varinjlim \Hom( -, D)$.  The category of cone complexes is the full subcategory of the category of set-valued functors on cones consisting of such colimits.  Since the functor from cones to topological spaces is faithful, and topological spaces admit finite colimits, this functor extends naturally to a faithful functor from cone complexes to topological spaces, taking a cone complex $\varinjlim \Hom( -, D)$ to the topological colimit of the diagram of cones $D$.  In particular, morphisms of cone complexes in this functor category are determined by continuous maps of topological spaces constructed as colimits of finite diagrams of cones, as described above.

Similar remarks apply to our topological descriptions of the categories of generalized cone complexes and extended generalized cone complexes, described  in Section \ref{Sec:generalized} below.
\end{remark}

Note that cone complexes differ from the fans considered in the theory of toric varieties \cite[1.1]{Oda}, \cite[1.4]{FultonToric} in two essential ways.  First, unlike a fan, a cone complex does not come with any natural embedding in an ambient vector space.  Furthermore, while the intersection of any two cones in a fan is a face of each, the intersection of two cones in a cone complex may be a union of several faces.  The latter is similar to the distinction between simplicial complexes and $\Delta$-complexes in cellular topology.  See, for instance, \cite[Section~2.1]{Hatcher02}.

 Write $\sigma_i$  for the cones in $\Sigma$,  and $\sigma_i^\circ$ for the relative interiors.  By definition, $\sigma_i^\circ$ is the interior of $\sigma$ in $\Hom(M, \RR)$.  It is also the complement of the union of all faces of positive codimension in $\sigma$.  Every point in a cone $\sigma$ is contained in the relative interior of a unique face, and this generalizes to cone complexes in the evident way:

\begin{proposition}\label{Prop:complex-decomposition}
Let $\Sigma$ be a cone complex. Then $\Sigma = \sqcup \sigma_i^\circ$ is the disjoint union of the relative interiors of its faces.
\end{proposition}

\subsection{Extended cones}

Let $\sigma$ be a cone, with $S_\sigma \subset M$ the additive monoid of integral linear functions that are nonnegative on $\sigma$, so  $\sigma = \Hom(S_\sigma, \RR_{\geq 0})$.  The associated \emph{extended cone} is
\[
\osigma = \Hom(S_\sigma, \RR_{\geq 0} \sqcup \{\infty\}).
\]
It is a compact space containing $\sigma$ as a dense open subset.  If $\tau$ is a face of $\sigma$, then the closure of $\tau$ in $\osigma$ is canonically identified with the extended cone $\overline \tau$, and we refer to $\overline \tau$ as an \emph{extended face} of $\osigma$, since it is the extension of a face of $\sigma$. The complement $\osigma \smallsetminus \sigma$ is the union of the faces at infinity, defined as follows.

Let $\tau \succeq \tau'$ be faces of $\sigma$, and consider the locally closed subset $F(\tau,\tau')$ in $\osigma$, consisting of points $v$ such that, for $u \in S_\sigma$,
\begin{enumerate}
\item $\<u,v\>$ is finite if and only if $u$ vanishes on $\tau'$, and
\item $\< u, v \>$ is zero if and only if $u$ vanishes on $\tau$.
\end{enumerate}
The set $F(\tau, \tau')$ is naturally identified with the projection $\tau/\tau'$ of $\tau$ along the linear span of $\tau'$.  In particular, it is a cone of dimension $\dim \tau - \dim \tau'$, and we can speak of its faces and its relative interior accordingly.   Note that $F(\tau, 0)$ is $\tau$, and its closure $\overline{F}(\tau,0)$ is the extended face $\overline{\tau}$.  

If $\tau'$ is not zero, then $F(\tau, \tau')$ is disjoint from $\sigma$, and we refer to its closure $\overline{F}(\tau, \tau')$ as a \emph{face at infinity}.  We refer to both the extended faces and the faces at infinity as \emph{faces} of $\osigma$.

For arbitrary $\tau \succeq \tau'$, the face $\overline{F}(\tau,\tau')$ decomposes as a disjoint union
\[
\overline F(\tau, \tau') = \bigsqcup_{\tau \succeq \gamma \succeq \gamma' \succeq \tau'} F(\gamma, \gamma')^\circ,
\]
and is canonically identified with the extended cone $\overline{\tau/\tau'}$.  In particular, the extended cone $\osigma$ is the disjoint union of the relative interiors of the locally closed faces $F(\tau, \tau')$.

Note that $F(\tau, \tau)$ is a single point, for each face $\tau \preceq \sigma$.  There is a natural simplicial complex structure on the extended cone $\osigma$ with these points as vertices.  The maximal cells in this complex are the closures of the maximal cones in the barycentric subdivision of $\sigma$.  

A morphism of extended cones is a continuous map of topological spaces $f: \osigma \rightarrow \osigma'$ whose restriction to $\sigma$ is a cone morphism
\[
f|_\sigma: \sigma \rightarrow F(\gamma, \gamma'),
\]
for some pair of faces $\gamma \succeq \gamma'$ of $\sigma'$.  In particular, the inclusion of a face at infinity is a morphism of extended cones.

\subsection{Extended cone complexes}\label{Sec:extended}

The skeleton of a toroidal scheme without self-intersection is a topological space obtained by gluing extended cones along proper inclusions of extended faces.  

\begin{remark}
For the purposes of studying toroidal schemes and stacks, it is enough to consider spaces obtained by gluing  extended cones along extended faces, so this is the approach we follow.  For the study of more general spaces, such as the stable toric varieties appearing in \cite{Alexeev, HackingKeelTevelev06}, it would be natural to consider also complexes obtained by gluing extended cones along faces at infinity.
\end{remark}

An \emph{extended cone complex} is a topological space together with a finite collection of closed subspaces and an identification of each of these closed subspaces with an extended cone such that the intersection of any two of these extended cones is a union of extended faces of each.  A morphism of extended cone complexes is a continuous map of topological spaces $f: \oSigma \rightarrow \oSigma'$ such that, for each extended cone $\osigma$ in $\oSigma$, there is an extended cone $\osigma'$ in $\oSigma'$ such that $f|_\sigma$ factors through a morphism of extended cones $\osigma \rightarrow \osigma'$.  We refer to the extended faces and faces at infinity of the cones in $\oSigma$ as the extended faces and faces at infinity of $\oSigma$, respectively.

Note that the complement of the faces at infinity in an extended cone complex is a cone complex.  Conversely, to each cone complex $\Sigma$, we associate an extended cone complex $\oSigma$ obtained by gluing the extended cones $\osigma$ along the extended faces $\overline \tau$, whenever $\tau$ is a face of $\sigma$ in $\Sigma$.

\begin{example}
Let $\Sigma$ be the fan in $N_\RR$ corresponding to a toric variety $X$.  The extended cone complex $\oSigma$ is the skeleton of $X^\beth$, as studied in \cite{Thuillier}.  If $X$ is complete, then $\oSigma$ is also the extended tropicalization of $X$, as studied in \cite{Payne}.  For any fixed cone $\tau'$ in $\Sigma$, the cones $F(\tau, \tau')$ in $N_\RR / \mathrm{span}(\tau')$ form a fan corresponding to the closed torus invariant subvariety $V(\tau) \subset X$.

More generally, if $\Sigma$ is the cone complex corresponding to a toroidal variety $X$ without self-intersections, then the cones $F(\tau, \tau')$, for fixed $\tau'$, are naturally identified with the cones in the complex corresponding to the stratum $X_{\tau'}$, with its induced toroidal structure.
\end{example}

\begin{remark}
The preceding construction gives a faithful and essentially surjective functor from cone complexes to extended cone complexes, but this functor is not full.  There are additional morphisms of extended cone complexes that do not come from morphisms of cone complexes, sending extended faces in the domain into faces at infinity in the target.
\end{remark}

We write $\osigma^\circ$ for the complement of the union of all extended faces of positive codimension in the extended cone $\osigma$.  It is the union of the locally closed sets $F(\sigma, \tau)^\circ$ for $\tau \preceq \sigma$.  As in Proposition \ref{Prop:complex-decomposition}, an extended cone complex $\oSigma$ is the disjoint union over its extended faces $\osigma$ of the locally closed sets $\osigma^\circ$.

\subsection{Barycentric subdivisions.}
Each \emph{extremal ray}, or one-dimen\-sional face, of a cone $\sigma$ is spanned by a unique primitive generator, i.e. a point whose image in $\Hom(M, \RR)$ is a primitive lattice point in $\Hom(M, \ZZ)$.  The {\em barycenter} of $\sigma$ is the ray in its relative interior spanned by the sum of the primitive generators of extremal rays.  The iterated stellar subdivision along the barycenters of cones in $\Sigma$, from largest to smallest, produces the barycentric subdivision $B(\Sigma)$ of a cone complex $\Sigma$. See \cite[Example III.2.1]{KKMS}. The barycentric subdivision of any cone complex is simplicial, and isomorphic to a fan.  See \cite[Lemma 8.7]{Abramovich-Matsuki-Rashid}.

We define the \emph{barycentric subdivision $B(\oSigma)$ of the extended cone complex} $\oSigma$ to be the compact simplicial complex whose cells are the closures in $\oSigma$ of the cones in the barycentric subdivision of $\Sigma$.  Note that the barycentric subdivision $B(\oSigma)$ of $\oSigma$ is not the extended cone complex $\overline{B(\Sigma)}$ of the barycentric subdivision of $\Sigma$. For instance, if $\Sigma=\sigma=\RR_{\geq 0}^2$ is a single quadrant  the picture is given in Figure \ref{Fig:barycentric}.
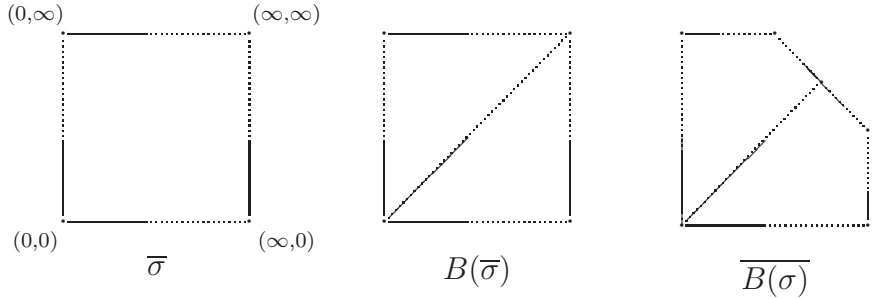
\begin{figure}[htb]
$$
\xymatrix@=0.15pc{*{\cdot}\ar@{-}[rrrr]^(-.3){(0,\infty)}\ar@{.}[rrrrrrrr] &&&&&&&&*{ \cdot}\ar@{}[r]^(1.6){(\infty,\infty)}\ar@{.}[dddddddd]&&&&&
\\ \\ \\ \\ &&&&&&&&\ar@{-}[dddd]\\ \\ \\ \\ 
*{ \cdot}\ar@{-}[uuuu]\ar@{.}[uuuuuuuu] 
\ar@{-}[rrrr]_(-.3){(0,0)}\ar@{.}[rrrrrrrr] &&&&&&&&
*{ \cdot}\ar@{}[r]_(1.6){(\infty,0)}&
\\
\ar@{}[rrrrrrrr]_{\textstyle \osigma}&&&&&&&&
}
\xymatrix@=0.15pc{*{ \cdot}\ar@{-}[rrrr]\ar@{.}[rrrrrrrr] &&&&&&&&*{ \cdot}\ar@{.}[dddddddd]\ar@{.}[ddddddddllllllll]&&&&
\\ \\ \\ \\ &&&&\ar@{-}[ddddllll]&&&&\ar@{-}[dddd]\\ \\ \\ \\ 
*{ \cdot}\ar@{-}[uuuu]\ar@{.}[uuuuuuuu]
\ar@{-}[rrrr]\ar@{.}[rrrrrrrr] &&&&&&&&
*{ \cdot}
\\
\ar@{}[rrrrrrrr]_{\textstyle B(\osigma)}&&&&&&&&
}
\xymatrix@=0.15pc{*{ \cdot}\ar@{-}[rr]\ar@{.}[rrrr] &&&&*{ \cdot}\ar@{.}[ddrr]&
\\ 
&&&&&\ar@{-}[dr]
\\
&&&&&&*{ \cdot}\ar@{{-}{-}{-}}[dr]\ar@{.}[drdr]\ar@{.}[ddddddllllll]\ar@{{-}{-}{-}}[lu] 
 \\&&&&&&&
  \\ &&&&\ar@{-}[ddddllll]&&&&*{ \cdot}\ar@{.}[dddd]
  \\ \\
  &&&&&&&&\ar@{-}[dd]
   \\ \\ 
*{ \cdot}\ar@{{-}{-}{-}}[uuuu]\ar@{.}[uuuuuuuu] 
\ar@{-}[rrrr]\ar@{.}[rrrrrrrr] &&&&&&&&
*{ \cdot}&
\\
\ar@{}[rrrrrrrr]_{\textstyle \overline{B(\sigma)}}&&&&&&&&
}
$$
\caption{Barycentric subdivisions}\label{Fig:barycentric}
\end{figure}



\subsection{Generalized cone complexes.}\label{Sec:generalized} In addition to cone complexes, we will consider spaces obtained as colimits of more general diagrams of cones, as follows.  A \emph{face morphism} of cones $\sigma \rightarrow \sigma'$ is an isomorphism onto a face of $\sigma'$.  We emphasize that the image of a face morphism is not required to be a proper face of the target, so any isomorphism of cones is a face morphism.

A {\em generalized cone complex} is a topological space with a presentation as the colimit of an arbitrary finite diagram of cones with face morphisms; we emphasize that the diagram need not be a poset.  Suppose $D$ and $D'$ are finite diagrams of cones with face morphisms, and let $\Sigma = \varinjlim D$ and $\Sigma' = \varinjlim D'$ be the corresponding generalized cone complexes.  A morphism of generalized cone complexes $f: \Sigma \to \Sigma'$ is a continuous map of topological spaces such that, for each cone $\sigma \in D$ there is a cone $\sigma'$ in $D'$ such that the induced map $\sigma \rightarrow \Sigma'$ factors through a cone morphism $\sigma \rightarrow \sigma'$.

\begin{remark}
By construction, the category of generalized cone complexes is an extension of the category of cone complexes that contains colimits for arbitrary finite diagrams of cones with face morphisms and comes with a faithful functor to topological spaces that commutes with such colimits; see Remark~\ref{rem:colimits}.  In particular, one can glue a cone to itself along isomorphic faces or take the quotient of a cone by a subgroup of its automorphism group in the category of generalized cone complexes.
\end{remark}

\noindent Similar objects were named  stacky fans in \cite{BMV, Chan-Melo-Viviani}, but that term is also standard for combinatorial data associated to toric stacks \cite{Borisov-Chen-Smith}.

If $\sigma$ is a cone in a finite diagram $D$ of cones with face morphisms, the image of the open cone $\sigma^\circ$ in the generalized cone complex $\Sigma = \varinjlim D$ is not necessarily homeomorphic to an open cone.  Nevertheless, the space underlying  a generalized cone complex has a natural cone complex structure, induced from the barycentric subdivisions of the cones in $D$.  We call this cone complex the \emph{barycentric subdivision} $B(\Sigma)$ of $\Sigma$. 
 
Similarly, a \emph{generalized extended cone complex} is a topological space with a presentation as a colimit of a finite diagram of extended cones in which each arrow is an isomorphism onto an extended face of the target.  

The functor from cones to extended cones generalizes to this setting, with the natural functor taking the colimit $\Sigma$ of a diagram of cones with face maps to the colimit $\oSigma$ of the corresponding diagram of extended cones.  The barycentric subdivision $B(\Sigma)$ induces a simplicial complex structure $B(\oSigma)$ on $\oSigma$, in which the maximal cells are the closures of the maximal cones in $B(\sigma)$.  Again, this is not the same as the extended complex $\overline{B(\Sigma)}$ of the barycentric subdivision of $\Sigma$ (Figure \ref{Fig:barycentric-quotient}).

\begin{figure}[htb]
$$
\xymatrix@=0.15pc{&&&&&&&&*{ \cdot}\ar@{.}[dddddddd]\ar@{{-}{--}{-}}[ddddddddllllllll]\ar@{}[r]^(1.6){(\infty,\infty)}&&&&&
\\ \\ \\ \\ &&&&&&&&\ar@{-}[dddd]\\ \\ \\ \\ 
*{ \cdot} 
\ar@{-}[rrrr]_(-.3){(0,0)}\ar@{.}[rrrrrrrr] &&&&&&&&
*{ \cdot}\ar@{}[r]_(1.6){(\infty,0)}&
\\
\ar@{}[rrrrrrrr]_{\textstyle \osigma\,\big/\,(\ZZ/2\ZZ)}&&&&&&&&
}
\xymatrix@=0.15pc{&&&&&&&&*{ \cdot}\ar@{.}[dddddddd]\ar@{.}[ddddddddllllllll]&&&&
\\ \\ \\ \\ &&&&\ar@{-}[ddddllll]&&&&\ar@{-}[dddd]\\ \\ \\ \\ 
*{ \cdot}
\ar@{-}[rrrr]\ar@{.}[rrrrrrrr] &&&&&&&&
*{ \cdot}
\\
\ar@{}[rrrrrrrr]_{\textstyle B\big(\osigma\,\big/\,(\ZZ/2\ZZ)\big)}&&&&&&&&
}
\xymatrix@=0.15pc{ &&&&&
\\ 
\\
&&&&&&*{ \cdot}\ar@{{-}{-}{-}}[dr]\ar@{.}[drdr]\ar@{.}[ddddddllllll]
 \\&&&&&&&
  \\ &&&&\ar@{-}[ddddllll]&&&&*{ \cdot}\ar@{.}[dddd]
  \\ \\
  &&&&&&&&\ar@{-}[dd]
   \\ \\ 
*{ \cdot}
\ar@{-}[rrrr]\ar@{.}[rrrrrrrr] &&&&&&&&
*{\cdot}&
\\
\ar@{}[rrrrrrrr]_{\textstyle \overline{B\big(\sigma\,\big/\,(\ZZ/2\ZZ)\big)}}&&&&&&&&
}
$$
\caption{The barycentric subdivision of an extended generalized cone complex is \emph{not} the extended cone complex of the barycentric subdivision. The dashed line on the left indicates folding.} 
\label{Fig:barycentric-quotient}
\end{figure}
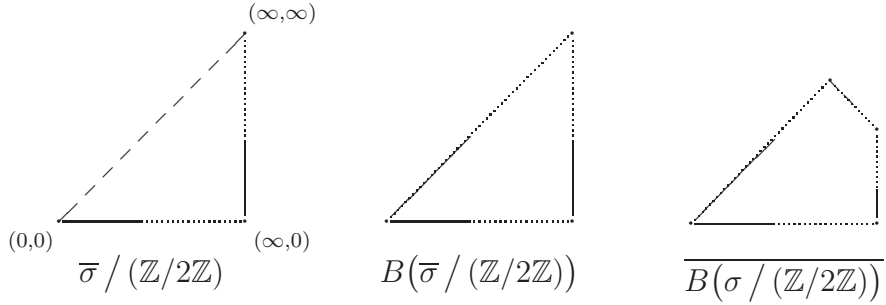

Some care is required to state the analogue of Proposition \ref{Prop:complex-decomposition} for generalized cone complexes and generalized extended cone complexes.  If $\Sigma = \varinjlim D$ is a generalized cone complex then there may be distinct cones in $D$ that are connected by isomorphisms, and there may also be arrows in $D$ that are nontrivial automorphisms.  We replace $D$ by a diagram with an isomorphic colimit in the category of generalized cone complexes, as follows.  First, add all faces of cones in $D$ with their associated inclusion maps to get a diagram $D'$.  Then choose a set of representatives $\{\sigma_i\}$ of the equivalence classes of cones in $D'$, under the equivalence relation generated by setting $\sigma \sim \sigma'$ if there is an isomorphism from $\sigma$ to $\sigma'$ in $D'$.  Now consider the diagram $D''$ whose objects are these representatives, and whose arrows are all possible maps obtained as compositions of arrows in $D'$ and their inverses.  Note that the set of self maps $\sigma_i \to \sigma_i$ in $D''$ is a subgroup $H_i$ of $\Aut(\sigma_i)$.  Then a point of $\Sigma$ is in the image of the relative interior $\sigma_i^\circ$ of a unique cone $\sigma_i$, and two points in  $\sigma_i^\circ$  have the same image if and only if they are identified by the diagram, namely they are in the same orbit of $H_i$.

We say that a finite diagram of cones with face morphisms is {\it reduced} if every face of a cone in the diagram is in the diagram, all isomorphisms are self-maps, and all compositions of arrows in the diagram are included.  The construction above shows that every generalized cone complex is the colimit of a reduced diagram of cones with face morphisms.  The correct analogue of Proposition \ref{Prop:complex-decomposition} is the following:

\begin{proposition}\label{Prop:generalized-decomposition}
Let� $\Sigma = \varinjlim D$ be a generalized cone complex.
\begin{enumerate}
\item There is a reduced diagram of cones with face morphisms $D^\mathrm{red}$ such that $\varinjlim D^\mathrm{red} \cong \Sigma$.
\item If $D$ is a reduced diagram of cones with face morphisms then
\[
\Sigma = \bigsqcup \sigma_i^\circ / H_i \mathrm{\ and \ } \oSigma = \bigsqcup \osigma_i /H_i,
\]
where the union is over all cones in $D$, and $H_i$ is the group of arrows from $\sigma_i$ to itself in $D$.
\end{enumerate}
\end{proposition}

\section{Algebraic curves, dual graphs, and moduli} \label{Sec:curves}

\subsection{Stable curves}

Fix an algebraically closed  field $k$. An {\em $n$-pointed nodal curve $(C;p_1,\ldots,p_n)$ of genus $g$} over $k$ is a  projective curve $C$ with arithmetic genus $g=g(C)$ over $k$ with only nodes as possible singularities, along with $n$ ordered distinct smooth points $p_i\in C(k)$. The curve is {\em stable} if it is connected and the automorphism group $\Aut(C,p_i)$ of $C$ fixing the points $p_i$ is finite. 

A curve $C$ over any field is said to be stable if the base change to the algebraic closure is stable.

\subsection{The dual graph of a pointed curve} \label{Sec:dualgraph}

The material here can be found with slightly different notation in \cite{Arbarello-Cornalba-Griffiths,C2}. See also \cite{BMV}.

Recall that to each $n$-pointed curve $(C;p_1,\ldots,p_n)$ with at most nodes as singularities over an algebraically closed field one assigns its  {\em weighted dual graph}, written somewhat succinctly as
$$\WG_C\ =\ \WG\ =\ (V, E, L, h),$$
where 
\begin{enumerate}
\item the set of vertices $V = V(\WG)$ is the set of irreducible components of $C$;
\item the set of edges  $E = E(\WG)$ is the set of nodes of $C$, where an edge $e\in E$ is incident to vertices $v_1,v_2$ if the corresponding node lies in the intersection of the corresponding components;
\item the {\em ordered} set of {\em legs} of  $L = L(\WG)$ correspond to the marked points, where a marking is incident  to the component on which it lies.
\item the function $h: V \to \NN$ is the genus function, where $h(v)$ is the geometric genus of the component corresponding to $v$.
\end{enumerate}
Note that a node of $C$ that is contained in only one irreducible component corresponds to a loop in $G_C$.

\begin{figure}[ht]
$$
\xymatrix@=0.5pc
{
&&&&&&&&&&\\
 && &     
*{\bullet}
\ar@{{-}{-}{-}}[lllu]_(0.8){l_1}_(0,1){1}|-(1.1){.}|-(1.2){.}|-(1.3){.}
\ar@{{-}{-}{-}}[llld]^(0.8){l_2}|-(1.1){.}|-(1.2){.}|-(1.3){.}   
\ar @{{-}{-}{-}} @/_.6pc/[rrrr] 
\ar@{{-}{-}{-}}@(dr,dl)
\ar@{{-}{-}{-}} @/^.6pc/[rrrr]
&&&&*{\bullet}\ar@{{-}{-}{-}}|-(1.1){.}|-(1.2){.}|-(1.3){.}  [rrrd]_(0.8){l_4}\ar@{{-}{-}{-}}|-(1.1){.}|-(1.2){.}|-(1.3){.}  [rrru]^(0.8){l_3}^(0,1){2}
\ar@{{-}{-}{-}}@(dr,dl)
\\
&&&&&&&&&&\\
}
$$\caption{A four-legged weighted graph of genus 6}\label{Fig:graph}
\end{figure}
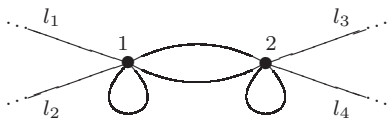

\begin{remark} As customary, the notation suppresses some data, which are nevertheless an essential part of $\WG$:
\begin{enumerate} \item The incidence relations between edges and vertices is omitted.
\item 
Consistently  
with \cite{BM,Arbarello-Cornalba-Griffiths,C2} 
we view an edge $e\in E(\WG)$ as a pair of distinct half-edges; the effect of this is that if $e$ is a loop, there is a nontrivial graph involution switching the two half edges of $e$.
\end{enumerate} 
When we talk about a {\em graph} we will always mean a weighted graph, unless we explicitly refer to the {\em underlying graph} of a weighted graph.
\end{remark}
The valence $n_v$  of a vertex $v\in V$ is the total number of incidences of edges and legs at $v$, where each loop contributes two incidences. The graph is said to be {\em stable} if it is  connected and satisfies the following: for every $v\in V$:
\begin{itemize} 
\item if $h(v) = 0$ then $n_v \geq 3$; and 
\item if $h(v) = 1$ then $n_v \geq 1$.
\end{itemize} 

\noindent Note that a pointed nodal curve $C$ is stable if and only if the graph $\WG_C$ is stable.

	The \emph{genus} $g(\WG)$ of a connected weighted graph is $$h^1(\WG,\QQ) + \sum_{v\in V} h(v),$$ and for a connected pointed nodal curve $C$ we have $$g(C) = g(\WG_C).$$ In essence this means that the weight $h(v)$ can be imagined as a replacement for $h(v)$ infinitely small loops hidden inside $v$, or even an arbitrary, infinitely small graph of genus $h(v)$ hidden inside $v$. 
	
	By the {\em automorphism group} $\Aut(\WG)$ of $\WG$ we mean the set of graph automorphisms preserving the ordering of the legs and the genus function $h$.

A {\em weighted graph contraction} $\pi:\WG \to \WG'$ is a contraction of the underlying graph  (composition of edge contractions), canonically endowed  with weight function $h'$  given by $$h'(v') = g(\pi^{-1}v'),$$ for $v'$ in $V(\WG')$.  Note that weighted graph contractions preserve the genus, and any contraction of a stable weighted graph is stable.

\subsection{Strata of the moduli space of curves} \label{Sec:strata} We consider the \emph{moduli stack} $\ocM_{g,n}$ and its \emph{coarse moduli space} $\oM_{g,n}$. The stack $\ocM_{g,n}$ is smooth and proper, and the boundary $\ocM_{g,n} \setminus \cM_{g,n}$ is a normal crossings divisor; this endows $\ocM_{g,n}$ with  a natural toroidal structure given by the open embedding $\cM_{g,n} \subset \ocM_{g,n}$. See \cite{Deligne-Mumford, Arbarello-Cornalba-Griffiths, Harris-Morrison} for generalites on moduli spaces, \cite{KKMS} for an introduction to toroidal embeddings, and Section \ref{Sec:toroidalDM} below for toroidal Deligne--Mumford stacks in general.

The toroidal structure on $\ocM_{g,n}$ induces a stratification, described as follows. Each stable graph $\WG$ of genus $g$ with $n$ legs  corresponds to a smooth, locally closed stratum $\cM_\WG\subset \ocM_{g,n}$.  The curves parametrized by $\cM_\WG$ are precisely those whose dual graph is isomorphic to $\WG$. The codimension of $\cM_\WG$ inside   $\ocM_{g,n}$ is the number of edges of $\WG$,  and 
$\cM_\WG$ is contained in the closure of $\cM_{\WG'}$ if and only if there is a graph contraction $\WG \to \WG'$.

\subsection{Explicit presentation of $\cM_\WG$} \label{sec:autg-cover} The stratum $\cM_\WG$ parametrizing curves with dual graph isomorphic to $\WG$ has the following explicit description in terms of moduli of smooth curves and graph automorphisms, see \cite[Section XII.10]{Arbarello-Cornalba-Griffiths}:

Recall that the valence of a vertex $v\in V(\WG)$ is denoted $n_v$. Consider
the moduli space $\widetilde \cM_\WG = \prod_v \cM_{h(v),n_v}$. The stack $\widetilde \cM_\WG$ can be thought of as the moduli stack of ``disconnected stable curves", where the universal family $\widetilde \cC^{dis}_\WG$ is the disjoint union of the pullbacks of the universal families $\cC_{h(v),n_v}\to \cM_{h(v),n_v}$. 
The disconnected curves parametrized by $\widetilde \cM_\WG$ have connected components corresponding to $V(\WG)$, no nodes,  and markings in the disjoint
union $\widetilde L = \sqcup_v\{p^v_1,\ldots,p^v_{n_v}\}$. The data of the graph $\WG$ indicate a gluing map $\widetilde \cC^{dis}_\WG \to \widetilde \cC_\WG$, and $\widetilde \cC_\WG\to \widetilde \cM_\WG$ is a family of connected curves, with irreducible components identified with $V(\WG)$, marked points  identified with $L(\WG)\subset \widetilde L$, nodes  identified with  $E(\WG)$ and branches of nodes identified with $\widetilde E = \widetilde L \setminus L(\WG)$.  Indeed, the family of glued curves exhibits $\widetilde \cM_\WG$ as the moduli space of curves with graph identified with $\WG$. There are sections $\widetilde \cM_\WG\to \widetilde \cC_\WG$ landing in the nodes, which are images of the sections $\widetilde \cM_\WG\to \widetilde \cC^{dis}_\WG$ determining the branches.

The group $\Aut(\WG)$ acts on $\widetilde \cC_\WG\to \widetilde \cM_\WG$ giving a map $[\widetilde \cM_\WG / \Aut(\WG)] \to \cM_\WG$ such that $[\widetilde \cC_\WG/\Aut(\WG)] \to [\widetilde \cM_\WG / \Aut(\WG)]$ is the pullback of the universal family  $\cC_\WG \to \cM_\WG$.

\begin{proposition}\label{Prop:stratum-quotient}
The quotient stack  $\left[\widetilde \cM_\WG / \Aut(\WG)\right]$ is canonically isomorphic to $\cM_\WG$, and $\left[\widetilde \cC_\WG/\Aut(\WG)\right]$ is canonically isomorphic to $\cC_\WG$.
\end{proposition}
\begin{proof} In  \cite[Proposition XII.10.11]{Arbarello-Cornalba-Griffiths} one  obtains a description of the compactification $\left[\, \overline{\widetilde \cM_\WG}\, /\, \Aut(\WG)\,\right]$ as the {\em normalization} of the closure $\overline{\cM_\WG}$ of $\cM_\WG$ in $\ocM_{g,n}$.  Since the open moduli space $\cM_\WG$ is already normal, and since $\left[\widetilde \cC_\WG/\Aut(\WG)\right]$ is the universal family, the proposition follows.\end{proof}

\section{Tropical curves and their moduli}
 \subsection{Tropical curves and extended topical curves}

A {\em tropical curve} is a metric weighted graph $$\TC = (\WG,\ell) = (V,E,L,h,\ell),$$ where  $\ell: E \to \RR_{>0}$.  

One can realize a tropical curve as an ``extended'' metric space  (keeping the weights on the vertices) 
by realizing an edge $e$ as an interval of length $\ell(e)$,
$$\xymatrix{
*{{v_1^{\vphantom{x}}\,}\bullet}\ar@{-}[rrrr]_{\ell(e)} &&&& *{\bullet{\,\,v_2^{\vphantom{x}}}}  
}$$
 and realizing a leg as a copy of $\RR_{\geq 0} \sqcup \{\infty\}$ where  $0$ is attached to its incident vertex:
$$\xymatrix{
*{{v\,}\bullet} \ar@{-}[rrr] && \ar@{.}[rrrr]_(1){\infty} &&&&*{\scriptstyle{\bullet}}
}$$  
Note that the infinite point on a leg of a  tropical curve is a distinguished point which does not correspond to a vertex. Removing these infinite points gives a usual metric graph which is not compact.

We identify $\Aut(\TC) \subset \Aut(\WG)$ as the subgroup of symmetries preserving the length function  $\ell$.

An {\em extended tropical curve} is an {\em extended metric weighted graph} $\TC = (\WG,\ell) = (V,E,L,h,\ell)$, where this time $\ell: E\to \RR_{>0}\sqcup \{\infty\}$; we realize an extended tropical curve as an extended metric space by realizing an edge $e$ with $\ell(e) = \infty$ as 
$$\left(\RR_{\geq 0}\sqcup \{\infty\}\right) \cup \left(\{-\infty\}\sqcup \RR_{\leq 0}\right)$$ where the points at infinity are identified:
$$\xymatrix{
*{{v_1^{\vphantom{x}}\,}{\bullet}} \ar@{-}[rr]&&*{}  \ar@{.}[rrr]_(.33){\infty}&*{\scriptstyle{\bullet}}&*{}\ar@{-}[rr] && *{\bullet{\,\,v_2^{\vphantom{x}}.}}  
}$$
We again realize a leg of an extended tropical curve as a copy of $\RR_{\geq 0}\sqcup\{\infty\}$, with $0\in \RR_{\geq 0}\sqcup\{\infty\}$ attached to its incident vertex.

Two tropical curves are said to be (tropically)  equivalent if they can be obtained  from one another by adding or removing  $2$-valent vertices of weight $0$ (without altering the underlying metric space),  or   $1$-valent vertices of weight $0$  together with their adjacent edge.
When
studying moduli of tropical curves with $2g-2+n>0$ we can (and will) restrict our attention to curves whose
 underlying weighted   graph $\WG$ is stable, as in every equivalence class of tropical curves  there is a unique representative whose underlying weighted graph is stable, see \cite[Sect. 2]{C2} for details; sometimes   such tropical curves are referred to as ``stable".

\subsection{Moduli of tropical curves: fixed weighted graph}
 
 The open cone of dimension $|E|$
$$\sigma^\circ_\WG\ = \ (\RR_{> 0})^{E}$$
parametrizes tropical curves together with an identification of the underlying graph with $\WG$, where
each coordinate determines the length of the corresponding edge.  There is a natural universal family over $\sigma^\circ_\WG$ (see for instance \cite[Section~5]{LPP}), so we view it as a fine moduli space for tropical curves whose underlying graphs are identified with $\WG$.

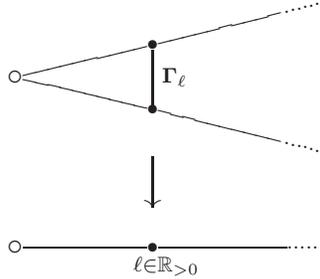
\begin{figure}[htb]
$$
\xymatrix@=0.5pc
{
&&&&&&&&&&\\
&&&& *{\scriptstyle{\bullet}}\ar@{-}[dd]^{\TC_\ell}
\ar@{-}|-(1.05){.}|-(1.1){.}|-(1.15){.}|-(1.2){.}|-(1.25){.}|-(1.3){.}[rrrru]
\ar@{-}[lllld]&&&&&\\
*{\circ}\ar@{-}[rrrrd]\\
&&&& *{\scriptstyle{\bullet}}\ar@{-}|-(1.05){.}|-(1.1){.}|-(1.15){.}|-(1.2){.}|-(1.25){.}|-(1.3){.} [rrrrd]\\
&&&&\ar[dd] &&&&\\ 
&&&\\
&&&&&&&& \\
*{\circ}\ar@{-}[rrrr]&&&& *{\scriptstyle{\bullet}} \ar@{-}|-(1.2){......}[rrrr]_(0,1){\ell\in \RR_{>0}} &&&&&\\
 \\
}
$$
$$\text{Example:} \xymatrix@=0.5pc
{\WG=&*{\bullet}\ar@{-}[rr] && *{\bullet} &&&\sigma^\circ_\WG=\RR_{>0}.}$$
\end{figure}

Tropical curves whose underlying graphs are isomorphic to $\WG$ are parametrized by
\[
M^\trop_\WG = \sigma^\circ_\WG / \Aut (\WG),
\]
since the identification of the underlying graph with $\WG$ is defined only up to automorphisms of $\WG$.  Note that $M^\trop_\WG$ is not even homeomorphic to an open cone, in general.  However, the set $\sigma^\circ_\WG / \Aut(\WG)$ is naturally identified with a union of relative interiors of cones in a cone complex, as follows.  The action of $\Aut(\WG)$ permutes the cones in the barycentric subdivision of $\WG$, and no point is identified with any other point in the same cone.  Therefore, each cone in the barycentric subdivision $B(\sigma_\WG)$ is mapped bijectively onto its image in the quotient.  This induces a cone complex structure on
\[
B(\sigma_\WG) / \Aut(\WG).
\]
Then $\sigma^\circ_\WG /\Aut(\WG)$ is the union of the relative interiors of those cones in $B(\sigma_\WG) / \Aut(\WG)$ that are images of cones in $B(\sigma_\WG)$ meeting $\sigma^\circ_\WG$.  Note that in general the universal family of $\sigma^\circ_\WG$ does not descend to the quotient, but there is a natural complex over $\sigma_\WG/ \Aut(\WG)$ whose fiber over a point $[\TC]$ is canonically identified with $\TC/ \Aut(\TC)$.   See Section~\ref{Sec:forgetful}.

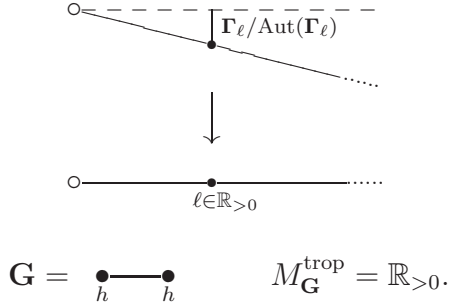
\begin{figure}[htb]
$$
\xymatrix@=0.5pc
{
*{\circ}\ar@{-}[rrrrd]\ar@{--}[rrrrrrrrr]
&&&& *{\scriptscriptstyle{}}\ar@{-}[d]^{\TC_\ell/\Aut(\TC_\ell)}&&&&&\\
&&&& *{\scriptstyle{\bullet}}\ar@{-}|-(1.05){.}|-(1.1){.}|-(1.15){.}|-(1.2){.}|-(1.25){.}|-(1.3){.}[rrrrd] \\
&&&&\ar[dd] &&&&\\ \\&&&&&&&& \\
*{\circ}\ar@{-}[rrrr]&&&& *{\scriptstyle{\bullet}} \ar@{-}|-(1.2){......}[rrrr]_(0,1){\ell\in \RR_{>0}}  &&&&\\
\\
}
$$
$$\xymatrix@=0.5pc
{\WG=&*{ \bullet}\ar@{-}[rr]_(0){h}_(1){h} && *{\bullet}&&&M^\trop_\WG=\RR_{>0}.}$$
 \end{figure}

Similarly, the extended cone $\osigma^\circ_\WG = (\RR_{> 0}\sqcup\{\infty\})^{E}$ is a fine moduli space for \emph{extended} tropical curves whose underlying graph is identified with $\WG$, and the quotient
\[
\oM_\WG^\trop = \osigma^\circ_\WG / \Aut(\WG)
\]
coarsely parametrizes extended tropical curves whose underlying graph is isomorphic to $\WG$.

\subsection{Moduli of tropical curves: varying graphs}\label{Sec:moduli-trop-glue}

Here we construct the moduli of tropical curves by taking the topological colimit of a natural diagram of cones, in which the arrows are induced by contractions of stable weighted graphs.  This approach is not original; it is quite similar, for instance, to the constructions in \cite{Kozlov2}.

As one passes to the boundary of the cone $\sigma^\circ_\WG$ parametrizing tropical curves whose underlying graph is identified with $\WG$, the lengths of some subset of the edges go to zero.  The closed cone $\sigma_\WG$ then parametrizes tropical curves whose underlying graph is identified with a \emph{weighted contraction} of $\WG$, as defined in Section~\ref{Sec:dualgraph}.  If $\varpi: \WG \rightarrow \WG'$ is a weighted contraction, then there is a canonical inclusion 
\[
\jmath_\varpi: \sigma_{\WG'} \hookrightarrow \sigma_{\WG}
\]
identifying $\sigma_{\WG'}$ with the face of $\sigma_{\WG}$ where all edges contracted by $\varpi$ have length zero. The cone $\sigma_\WG$ has a natural integral structure, determined by the integer lattice in $(\RR_{\geq 0})^{E}$ parametrizing tropical curves with integer edge-lengths.

As a topological space, the coarse moduli space of tropical curves $M_{g,n}^\trop$ is the colimit of the diagram of cones $\sigma_\WG$ obtained by gluing the cones $\sigma_{\WG}$ along the inclusions $\jmath_\varpi$ for all weighted contractions $\varpi$:
\begin{equation}
\label{Eq:limit} M_{g,n}^\trop = \varinjlim \left(\sigma_{\WG}, \jmath_\varpi\right).
\end{equation}
It is therefore canonically a generalized cone complex. Note that every automorphism of a weighted graph $\WG$ is a weighted contraction, so the map from $\sigma_\WG$ to the colimit $M_{g,n}^\trop$ factors through $\sigma_\WG / \Aut(\WG)$.  Furthermore, two points are identified in the colimit if and only if they are images of two points in some open cone $\sigma^\circ_\WG$ that differ by an automorphism of $\WG$.  Therefore, $M_{g,n}^\trop$ decomposes as a disjoint union
\[
M_{g,n}^\trop =  \bigsqcup_\WG M_\WG^\trop =  \bigsqcup_\WG \sigma_\WG^\circ/\Aut(\WG),
\]
over isomorphism classes of stable weighted graphs of genus $g$ with $n$ legs.  This is not a cell decomposition, but $M_{g,n}^\trop$ does carry a natural cone complex structure, induced from the barycentric subdivisions of the cones $\sigma_\WG$, in which each $M_\WG^\trop$ is a union of relative interiors of cones.  There is also a ``universal family," whose fiber over a point $[\TC]$ is canonically identified with $\TC/ \Aut(\TC)$.  Again, see Section~\ref{Sec:forgetful}.

Similarly, the coarse moduli space of extended tropical curves $\oM_{g,n}^\trop$ is the generalized extended cone complex 
\[
\oM_{g,n}^\trop = \varinjlim \left(\osigma_{\WG}, \jmath_\varpi\right),
\]
which decomposes as a disjoint union
\[
\oM_{g,n}^\trop =  \bigsqcup_\WG \oM_\WG^\trop =  \bigsqcup_\WG \osigma_\WG^\circ/\Aut(\WG).
\]

\begin{remark}
As in Section \ref{Sec:complexes}, while $M_{g,n}^\trop$ inherits a cone complex structure from its barycentric subdivision $B(M_{g,n}^\trop)$,  the compactification $\oM_{g,n}^\trop$ has a simplicial structure in the form $B(\oM_{g,n}^\trop)$, which is not the same as the associated extended cone complex $\overline{B(M_{g,n}^\trop)}$.
\end{remark}

\section{Thuillier's skeletons of toroidal schemes}\label{Sec:Thuillier}

Here we recall the basic properties of cone complexes associated to toroidal embeddings, and Thuillier's treatment of their natural compactifications as analytic skeletons.

\subsection{Thuillier's retraction}

We begin by briefly discussing Thuillier's construction of the extended cone complex and retraction associated to a toroidal scheme without self-intersection, an important variation on Berkovich's fundamental work on skeletons of pluristable formal schemes \cite{Berkovich-contraction}.  Recall that a toroidal scheme is a pair $U \subset X$ that \'etale locally looks like the inclusion of the dense torus in a toric variety: every point $p\in X$ has an \'etale neighborhood $\alpha:V \to X$ which admits an \'etale map $\beta:V \to V_\sigma$ to an affine 
toric variety, such that $\beta^{-1}T = \alpha^{-1} U$, where $T\subset V_\sigma$ is the dense torus.   It is a toroidal embedding  without self-intersection if each irreducible component of the boundary divisor $X \smallsetminus U$ is normal, in which case $V$ can be taken to be a Zariski open subset of $X$.  For further details, see  \cite{KKMS,Thuillier}.

\begin{remark}
Thuillier defines toroidal embeddings  in terms of \'etale charts, whereas in \cite{KKMS} they are defined in terms of formal completions. In \cite[IV.3.II, p. 195]{KKMS} the approaches are shown to be equivalent for toroidal embeddings without self-intersection. A short argument of Denef \cite{Denef} shows that the approaches are equivalent in general.
\end{remark}

We work over an algebraically closed field $k$, equipped with the trivial valuation, which sends $k^*$ to zero. The usual Berkovich analytic space associated with a variety $X$ over $k$ is denoted $X^\an$.  One also associates functorially another nonarchimedean analytic space in the sense of Berkovich, denoted $X^\beth$; here $\scriptstyle\beth$ is the Hebrew letter {\em bet}:  

\begin{definition}\label{Def:beth}
The space $X^\beth$ is the compact analytic domain in $X^\An$ whose $K$-points, for any valued extension $K | k$ with valuation ring $R \subset K$, are exactly those $K$-points of $X$ that extend to $\Spec R$.
\end{definition}

\noindent In particular, we have natural identifications 
\[
X^\beth(K) = X(R),
\]
for all such valued extensions.  If $X$ is proper, then every $K$-point of $X$ extends to $\Spec R$, and $X^\beth$ is equal to $X^\An$. See \cite[Sections 1.2, 1.3]{Thuillier}.

\begin{example}
Let $X$ be a toric variety with dense torus $T$, corresponding to a fan $\Sigma$ in $N_\RR$.  A $K$-point $x$ of $T$ extends to a point of $X$ over $\Spec R$ if and only if $\Trop(x)$ is contained in the fan $\Sigma$.  In other words, $X^\beth \cap T^\an$ is precisely the preimage of $\Sigma$ under the classical tropicalization map.  The extended tropicalization map from \cite{Payne} takes $X(K)$ into a partial compactification of $N_\RR$, and the closure of $\Sigma$ in this partial compactification is the extended cone complex $\oSigma$.  The preimage of $\oSigma$ under the extended tropicalization map is exactly $X^\beth$.
\end{example}

 Given a toroidal embedding $U \subset X$ over $k$, Thuillier defines a natural continuous, but not analytic, idempotent self-map 
 $\boldp_X:X^\beth \to X^\beth$.
 \begin{definition}
 The \emph{skeleton} $\oSigma(X) \subset X^\beth$ is the image of  the
 map
 $\boldp_X$.
 \end{definition}
 
\noindent The map  $\boldp_X$ is referred to as the {\it retraction} of $X^\beth$ to its skeleton; we write simply  $\boldp$ when no confusion seems possible.

If $U \subset X$ is a toroidal embedding without self-intersection then the image of $U^\an \cap X^\beth$ is canonically identified with the cone complex $\Sigma(X)$ associated to the toroidal embedding, as constructed in \cite{KKMS}.  Then $\oSigma(X)$ is the closure of $\Sigma(X)$ in $X^\beth$, and is canonically identified with the extended cone complex of $\Sigma(X)$. 
 The toroidal structure determines local monomial coordinates on each stratum of $X \smallsetminus U$, and the target of the retraction $\oSigma(X)$ is the space of monomial valuations in these local coordinates;  see \cite[Section~1.4]{BFJ} for details on monomial valuations.   Thuillier shows, furthermore, that $\boldp$ is naturally a deformation of the identity mapping on $X^\beth$, giving a canonical strong deformation retraction of $X^\beth$ onto $\oSigma(X)$. 
 
\begin{remark} \label{rem:strata}
There is a natural order reversing bijection between the strata of the boundary divisor $X \smallsetminus U$ and the cones in $\Sigma(X)$, generalizing the order reversing correspondence between cones in a fan and the boundary strata in the corresponding toric variety,  as follows.  Let $x$ be a point of $X^\beth$ over a valued extension field $K|k$ whose valuation ring is $R$.  Then $x$ is naturally identified with an $R$-point of $X$.  We write $\overline x$ for the reduction of $x$ over the residue field.  Then $\boldp(x)$ is contained in the relative interior $\sigma^\circ$ of a cone $\sigma$ if and only if the reduction $\overline x$ is in the corresponding locally closed boundary stratum $X_\sigma$ in $X$, over the residue field, and 
$x$ is in $U$.  In other words,  the preimage of $\sigma^\circ$ is the subset of $X^\beth \cap U^\An$ consisting of points over valued fields whose reduction lies in the corresponding stratum of $X \smallsetminus U$ over the residue field.
\end{remark}

\begin{remark}\label{rem:boundary}
The order reversing bijection between strata in $X \smallsetminus U$ and cones in $\Sigma(X)$ described in Remark~\ref{rem:strata}, which comes from the reduction map on $X^\beth$, should not be confused with the order preserving bijection between strata in $X \smallsetminus U$ and strata in $\oSigma(X) \smallsetminus \Sigma(X)$.  Quite simply, the preimage under $\boldp$ of a boundary stratum in $\oSigma(X) \smallsetminus \Sigma(X)$ is a stratum in the boundary divisor $(X \smallsetminus U)^\An$.
\end{remark}

\subsection{Explicit realization of the retraction} \label{sec:explicit-retraction}
In this section we describe Thuillier's retraction to the extended cone complex more explicitly in local coordinates, for a toroidal embedding $U\subset X$ without self-intersection.  

The toroidal scheme without self-intersections  $X$ is covered by Zariski open toric charts $$\xymatrix{V_\sigma& V\ar[l]_\beta \ar@{^(->}[r]^\alpha& X}.$$  Recall that the cone $\sigma$ can be described in terms of monoid homomorphisms, as follows.  Let $M$ be the group of Cartier divisors on $V$ supported in the complement of $U$, and let $S_\sigma \subset M$ be the submonoid of such Cartier divisors that are effective.  Then the cone $\sigma$ is the space of monoid homomorphisms to the additive monoid of nonnegative real numbers,
\[
\sigma = \Hom (S_\sigma, \RR_{\geq 0}),
\]
equipped with its natural structure as a rational polyhedral cone with integral structure.  The associated extended cone $\osigma$ is
\[
\osigma = \Hom(S_\sigma, \RR_{\geq 0} \sqcup \{\infty\}).
\]

Let $x$ be a point in $X^\beth$.  Then $x$ is represented by a point of $X$ over a valuation ring $R$ with valuation $\val$.  Let $\overline x$ be the reduction of $x$, which is a point of $X$ over the residue field of $R$, and let $V \subset X$ be an open subset that contains $\overline x$ and has a toric chart $V \rightarrow V_\sigma$.  Then $\boldp(x) \in \osigma$ is the monoid homomorphism that takes an effective Cartier divisor $D$ on $V$ with support in the complement of $U$ and local equation $f$ at $x$ to
\begin{equation} \label{Eq:description-p}
\boldp(x)(D) = \val(f). 
\end{equation}
This is clearly a monoid homomorphism to $\RR_{\geq 0} \sqcup \{\infty\}$, and it is nonnegative because $D$ is effective. It is also independent of the choice of chart, the choice of extension field over which $x$ is rational, and the choice of defining equation for $D$. See \cite[Lemma 2.8, Proposition 3.11]{Thuillier}.   This describes $\boldp$ as a projection to a natural extended cone complex; we now explain how this cone complex may be seen as a subset of $X^\beth$.

First, we may assume that the open set $V \subset X$ is affine. Then $V^\An$ is the space of valuations on the coordinate ring $k[V]$ that extend the given trivial valuation on $k$ \cite[Remark~3.4.2]{Berkovich}, and $V^\beth$ is the subspace of valuations that are nonnegative on all of $k[V]$.  Here, a valuation on a $k$-algebra $R$ is a map
\[
\val: R \rightarrow \RR \cup \{\infty\}
\]
such that $\val(fg) = \val(f) + \val(g)$ and $\val(f + g) \geq \min\{ \val(f), \val(g) \}$, for all $f$ and $g$ in $R$, $\val(0) = \infty$, and $\val(a) = 0$ for all $a \in k^\times$.  In particular, $\val(f)$ may be equal to $\infty$ for some nonzero $f \in R$.

Shrinking the toric variety $V_\sigma$, if necessary, we may assume that there is a point $x$ in $V$ mapping to a point $x_\sigma$ in the closed orbit $O_\sigma \subset V_\sigma$; note that the following construction is independent of the choice of $x$.  Since $V \rightarrow V_\sigma$ is \'etale, the completed local ring $\widehat \cO_{X,x}$ is identified with $\widehat \cO_{V_\sigma, x_\sigma}$, which is a formal power series ring
\[
k[[y_1, ..., y_r]][[S_\sigma]],
\]
with exponents in $S_\sigma$ and coefficients in a formal power series ring in $r$ parameters.  Here, $r$ is the dimension of $O_\sigma$.  For each function 
$f \in k[V]$, let $\sum_{u \in S_\sigma} a_u(f) z^u$ be the image of $f$ in this power series ring.  Then, to each point $v$ in $\osigma$, we associate the monomial valuation $\val_v: k[V] \rightarrow \RR \cup \{ \infty \}$ taking $f$ to
\begin{equation} \label{Eq:embedding-p}
\val_v(f) = \min \{ \langle u,v\rangle \ | \ a_u(f) \neq 0 \}.
\end{equation}
Since $v$ is in the extended dual cone of $S_\sigma$, the valuation  $\val_v$ is nonnegative on $k[V]$, so this construction realizes $\osigma$ as a subset of $V^\beth$.  As $V$ ranges over an open cover of $X$, the subsets $V^\beth$ cover $X^\beth$, and the union of the cones $\osigma$, one for each stratum in $X$, is the extended cone complex $\oSigma(X) \subset X^\beth$. 

\begin{remark}
From this description of $\boldp$, we see that $\boldp(x)$ is contained in the relative interior $\osigma^\circ$ 
of $\osigma$ as defined in Section \ref{Sec:extended} if and only if the reduction $\overline x$ is contained in the smallest stratum of $V$.  Also, $\boldp(x)$ is contained in the boundary $\osigma \smallsetminus \sigma$ if and only if $x$ itself is contained in the boundary of $V$.  These two properties of $\boldp$ determine both the order reversing correspondence between cones of $\oSigma(X)$ and strata of $X$, and the order preserving correspondence between boundary strata in $\oSigma(X) \smallsetminus \Sigma(X)$ and boundary strata in $X \smallsetminus U$, discussed in Remarks~\ref{rem:strata} and \ref{rem:boundary}, above.
\end{remark}

We recall that Thuillier also constructs a canonical homotopy
$H_V:V^\beth\times [0,1] \to V^\beth$, such that $$H_V
\times \{0\}=\text{id}_V\ :\ \ V^\beth \to V^\beth,$$ and
$$H_V \times \{1\} =  \boldp_V \ :\ \  V^\beth \to \oSigma_V,$$ giving a strong deformation retraction of $V^\beth$ onto the skeleton $\oSigma(V)\subset X^\beth$, and that this construction is functorial for  \'etale  morphisms of toroidal schemes. Our goal in  Section \ref{Sec:toroidalDM} is to show that a similar construction applies to toroidal Deligne--Mumford stacks. 

\subsection{Functoriality}

 A morphism $X \to Y$ of toroidal embeddings is {\em toroidal} if  for each $x\in X$ there is a commutative diagram 
  $$\xymatrix{V_\sigma\ar[d] & V_X \ar[l]\ar[r]\ar[d] & X\ar[d] \\ V_\tau& V_Y \ar[l]\ar[r] & Y
  }$$ 
  where the top row is a toric chart at $x$, the bottom row is a toric chart at $f(x)$, and the arrow $V_\sigma\to V_\tau$ is a {\em dominant} torus equivariant map of toric varieties; this is a so called toric chart for the morphism $X\to Y$. Toroidal morphisms were introduced in \cite[Section 1]{Abramovich-Karu}; separable toroidal morphisms coincide with the logarithmically smooth maps of \cite{KatoToric}. 
  
More generally, we say that a morphism $X\to Y$ as above is {\em sub-toroidal}
if     $V_{\sigma}\to V_{\tau}$ is only assumed to dominate a torus invariant subvariety of $V_{\tau}$.
A key example is when $X \to Y$ is the normalization of a closed toroidal stratum $X'\subset Y$. 

\begin{proposition} The formation of \  $\oSigma(X)$ is functorial for  sub-toroidal morphisms:
If $f:X \to Y$ is a  sub-toroidal  morphism of toroidal embeddings without self-intersections, and $f^\beth:X^\beth \to Y^\beth$ is the associated morphism of Berkovich spaces, then $f^\beth$ restricts to a map of generalized extended cone complexes $\oSigma(f): \oSigma(X) \to \oSigma(Y)$. In particular  $\boldp_Y\circ f^\beth =\oSigma(f) \circ \boldp_X$, and if  $Y \to Z$ is another  sub-toroidal  morphism, then $\oSigma(g)\circ\oSigma(f) = \oSigma(g\circ f)$.
\end{proposition}

\begin{proof}
We first prove the result for toroidal morphisms, and then indicate the changes necessary for   sub-toroidal morphisms. 

The result is local, so we may assume there is a toric chart for $f$ which covers $X$ and $Y$, in which case $\oSigma(X)  =\osigma$ and $\oSigma(Y)  =\otau$ . If $S_\sigma$ and $S_\tau$ are the monoids of effective Cartier divisors on $X$ and $Y$ supported away from $U_X$ and $U_Y$, then we have a pullback homomorphism $S_\tau \to S_\sigma$, which is evidently compatible with   composition with a further toric morphism $Y \to Z$. This induces a map $\Sigma(f): \osigma \to \otau$ compatible with compositions. By Equation 
\eqref{Eq:description-p} of Section \ref{sec:explicit-retraction} we have   $\boldp_Y\circ f^\beth =\oSigma(f) \circ \boldp_X$. By Equation \eqref{Eq:embedding-p} we also have  $\oSigma(f) = f^\beth|_{\osigma}$, as needed.

For   sub-toroidal  morphisms we only need to replace $S_\sigma$ by $S_\sigma' :=S_\sigma\sqcup\{\infty\}$ and similarly for $S_\tau'$. We define $f^*: S_\tau' \to S_\sigma'$ by declaring that $f^*(\infty) = \infty$ and,  if $f(X) \subset D$ 
for some nonzero divisor $D\in S_\tau'$,
then also $f^*D = \infty$. The induced map $\osigma \to \otau$ maps $\osigma$ to a face at infinity of $\otau$ (see \cite[Proposition 2.13]{Thuillier}), and the rest of the proof works as stated. 
\end{proof}

\begin{remark}
If the map $V_\sigma \rightarrow V_\tau$ in the definitions of toroidal and sub-toroidal is only assumed to be torus equivariant, but not necessarily dominating a torus invariant subvariety, then the restriction of $f^\beth$ does not necessarily map $\oSigma(X)$ into $\oSigma(Y)$.  However, composing $f^\beth$ with $\boldp_Y$ still induces a map of generalized extended cone complexes, factoring through $\boldp_X$, and the argument above shows that this construction is functorial with respect to such morphisms.
\end{remark}

\section{Skeletons of toroidal Deligne--Mumford stacks} \label{Sec:toroidalDM}
Here we generalize Thuillier's retraction of the analytification of a toroidal
scheme onto its canonical skeleton  to the case of toroidal Deligne--Mumford
stacks. We follow the construction of \cite[Section 3.1.3]{Thuillier}, where toroidal embeddings with self-intersections are treated.

\subsection{Basic construction} Let $\cX$ be a separated connected Deligne--Mumford stack over $k$, with coarse moduli space $X$, which is a separated algebraic space, by \cite[Theorem~1.1]{KeelMori}. Let $U \subset \cX$ be an open substack and, for any morphism $V \rightarrow \cX$, let $U_V \subset V$ be the preimage of $U$.

\begin{definition} The inclusion $U \subset \cX$ is a {\em toroidal embedding of  Deligne--Mumford stacks}
  if, for every \'etale morphism from a scheme $V \to \cX$, the inclusion $U_V \subset V$ is a toroidal embedding of schemes.
  \end{definition}
  
When $U$ is understood, we refer to $\cX$ as a {\em toroidal
  Deligne--Mumford stack}. The property of being a toroidal embedding is \'etale local on schemes, so the inclusion $U \subset \cX$ is a toroidal embedding if and only if,  for a single \'etale covering $V \to \cX$, the embedding $U_V \subset V$ is toroidal.

Let $\cX$ be a toroidal Deligne--Mumford stack with coarse moduli space $X$, and  $V \to \cX$ 
an \'etale covering by a scheme, where 
$U_V \subset V$
 is a toroidal embedding without self-intersections. We write $V_2 = V\times_\cX V$. Then
$V_2 \double V \to X$ is a right-exact diagram of algebraic spaces.

The analytification $X^\an$ of a separated algebraic space $X$ is defined in \cite{Conrad-Temkin}. Using the same notation as in Definition  \ref{Def:beth}  we set the following:

\begin{definition}
Let $X$ be a separated algebraic space.  The space $X^\beth$ is the compact analytic domain in $X^\An$ whose $K$-points, for any valued extension $K | k$ with valuation ring $R \subset K$, are exactly those $K$-points of $X$ that extend to $\Spec R$.
\end{definition}

\begin{lemma}
The induced diagram $V_2^\beth \double V^\beth \to X^\beth$ is right-exact on the underlying topological spaces.
\end{lemma}

\begin{proof}
First we claim that this is a right exact diagram of sets. We use the description of the point-set of $X^\beth$ as the set of equivalence classes of $X(R)$, where $R$ runs over valuation rings extending $k$ with algebraically closed valued fraction fields $K$, and equivalence is determined by extensions of such valuation rings.   Let $R$ be such a valuation ring, and $x \in X(R)$. Since X is the coarse moduli space of $\cX$ and $K$ is algebraically closed, there is a lift  $\tilde x \in \cX(R)$ of $x$, unique up to isomorphisms.  Consider the base change along $\tilde x: \Spec R \rightarrow \cX$,
\[
(V_2 \times_\cX \Spec R) \, \double \, (V \times_\cX \Spec R) \, \to \Spec R
\]
The special fiber $\overline x$ of $x$ is a point of $X$ over the residue field $\kappa$ of $R$.  Since $X$ is the coarse moduli space of $\cX$ and $\kappa$ is algebraically closed, $X(\kappa)$ is the quotient of $V(\kappa)$ by the equivalence relation $V_2(\kappa)$.  In particular, $\overline x$ lifts to a point $\overline{y} \in V(\kappa)$.  Since $V \rightarrow \cX$ is flat, $(V \times_\cX \Spec R) \, \to \Spec R$ is also flat, and hence $\overline y$ is the special fiber of a point $y$ in $V(R)$.  Therefore, $x$ is in the image of $V(R)$, and $V^\beth \rightarrow X^\beth$ is surjective.  A similar argument shows that any pair of points in $V^\beth$ that map to the same point in $X^\beth$ are in the image of $V_2^\beth$, so $V_2^\beth \double V^\beth \to X^\beth$ is right exact on the underlying sets, as claimed

The topological spaces  $V_2^\beth$, $V^\beth$, and $X^\beth$ are all compact Hausdorff and the maps between them are continuous, since they are in fact analytic. The map $V^\beth \to X^\beth$ has finite fibers. If $Q$ is the topological quotient space of $V_2^\beth \double V^\beth$, then we have a continuous factorization $V^\beth \to Q \to  X^\beth$. So  $V^\beth \to Q$ also has finite fibers.  The equivalence relation determined by the image of $V^\beth_2$ in the topological product $V^\beth\times_{\mathfrak Top} V^\beth$ is closed, since $V_2^\beth$ is compact and $V^\beth$ is Hausdorff.   Since $Q$ is the quotient of the compact Hausdorff topological space $V^\beth$ by a closed equivalence relation with finite orbits it is also compact Hausdorff.  The continuous map $Q \to X^\beth$ is bijective by the previous paragraph. Since  $Q$ and $X^\beth$ are compact Hausdorff spaces, the map is a homeomorphism, as needed.
\end{proof}

\begin{prop}  \label{prop:Limit}
Let $\cX$ be a separated connected Deligne--Mumford stack with coarse moduli space $X$.  Then there is a canonical continuous map $H_\cX:[0,1] \times X^\beth \to
X^\beth$ connecting the identity to an idempotent self-map $\boldp_\cX$.
Writing $\oSigma(\cX)$ for the image of $\boldp_\cX$, we have a
commutative diagram of topological spaces

$$\xymatrix{V_2^\beth \doubleright \ar[d]_{\boldp_{V_2}} & V^\beth \ar[d]^{\boldp_V}\ar[r] & X^\beth\ar[d]^{\boldp_\cX} \\
                      \oSigma(V_2) \doubleright                               & \oSigma(V)\ar[r] & \oSigma(\cX),
                      }
                      $$
with right exact rows. In particular, $X^\beth$ is contractible and $\oSigma(\cX)$ is the topological colimit of the diagram $\oSigma(V_2) \double \oSigma(V)$. 
\end{prop}

\begin{proof}
Since $V^\beth_2 \double V^\beth \to X^\beth$ is topologically right exact it follows that the continuous map $H_X:[0,1]\times X^\beth \to X^\beth$ making the diagram 
$$\xymatrix{[0,1]\times {V}^\beth \ar[r] \ar[d]_{H_{V}} & [0,1]\times X^\beth \ar[d]^{H_X} \\
                      {V}^\beth \ar[r]                               & X^\beth
                      }
                      $$
commutative exists and is unique. To check that the map is independent of the  choice of $V$ it suffices to consider a different \'etale cover by a scheme $V' \rightarrow \cX$ that factors through $V \to \cX$. This induces a map $H'_X:[0,1]\times X^\beth \to X^\beth$ commuting with $H_{V'}$ as in the preceding diagram,
and a map $\boldp'_\cX:X^{\beth}\to \oSigma(\cX)$ commuting with $\boldp_{V'}$ as in the first diagram.
  By \cite[Lemma 3.38]{Thuillier} we also have a commutative diagram 
$$\xymatrix{[0,1]\times {V'}^\beth \ar[r] \ar[d]_{H_{V'}} & [0,1]\times V^\beth \ar[d]^{H_V} \\
                      {V'}^\beth \ar[r]                               & V^\beth.
                      }
                      $$
                      It follows that $H_X = H_X'$, 
and necessarily $\boldp'_\cX = \boldp_\cX$.

This shows that $X^\beth$ admits a deformation retraction onto $\oSigma(\cX)$.  It remains to show that $\oSigma(\cX)$ is contractible.  The natural cell complex structure on $\oSigma(\cX)$ has one vertex for each face of $\Sigma(\cX)$, and one $k$-face for each chain of length $k$ in the poset of faces.  In other words, $\oSigma(\cX)$ is homeomorphic to the geometric realization of the poset of faces of $\Sigma(\cX)$.  As $\cX$ is connected this poset has a minimal element, the unique zero-dimensional face of $\Sigma(\cX)$; its geometric realization is contractible, as required.
\end{proof}

\begin{remark}
One can define  $\oSigma(\cX)$ as the topological colimit of the diagram $\oSigma(V_2) \double \oSigma(V)$, describe it explicitly as in Proposition \ref{Prop:decomposition} and prove its functorial properties as in Proposition \ref{Prop:functoriality}, without recourse to Berkovich analytification, similarly to the work of \cite{KKMS} or \cite{Gross-Siebert}. However Part (2) of Theorem \ref{Th:functor}  requires the analytic context. 
\end{remark}

\begin{remark}
Citing a similar argument, Thuillier notes that the skeleton $\oSigma(X)$ of a toroidal embedding with self-intersection inherits the structure of a cell complex.  We mention one minor gap in the proof of this claim, which is easily corrected.  Lemma~3.33 in \cite{Thuillier} states that every cone of $\Sigma(V)$ maps isomorphically to its image in $\Sigma(X)$, which is not true in general, as seen in Example~\ref{ex:Monodromy}. Nevertheless, Thuillier's argument does show that cones of the barycentric subdivision of $\Sigma(V)$ map isomorphically to their images in $\Sigma(X)$, so the topological space $\Sigma(X)$ inherits the structure of a cone complex, induced from this barycentric subdivision.  The same holds for a toroidal Deligne--Mumford stack $\cX$: the skeleton $\oSigma(\cX)$ inherits a simplicial complex structure induced from the barycentric subdivision of $\Sigma(V)$.
\end{remark}

\begin{example} \label{ex:Monodromy}
Consider  $X = \AA^3\setminus \{y=0\}$ and let $U$ be the complement of the divisor  $$D = \{x^2y - z^2 = 0 \}.$$ Since $D$ is a normal crossings divisor, the inclusion $U \subset X$ is a toroidal embedding. Note that $D$ is irreducible but not normal, so $X$ is a toroidal variety with self-intersection.  Moreover, the preimage of the singular locus in the normalization of $D$ is irreducible.  Roughly speaking, this is because the fundamental group of the singular locus acts transitively on the branches of $D$ at the base point.  This phenomenon of \emph{monodromy} is discussed systematically in Section~\ref{Sec:monodromy}.

We now consider a toric chart on $X$.  Let  $ V = \AA^3\setminus \{u=0\}$ and  $$D_V=V(x^2u^2 - z^2)\ \subset\  V,$$ there is a degree 2 \'etale cover $V \to X$ given by $u^2=y$.  Now $V$ is a toroidal embedding without self-intersections.  The coordinate change $z_1=z/u$ gives the equation $x^2 = z_1^2$, so $V$ is isomorphic to $\AA^2 \times \GG_m$ with its standard toric structure - the toric divisors on $V$ are $\{x=z_1\}$ and $\{x=-z_1\}$.  The corresponding cone is $\sigma = \RR_{\geq 0}^2$.  The group $\ZZ/2\ZZ$ acts freely on $V$, by interchanging the sheets of the \'etale cover, with quotient $X$; the involution sends $(x,u,z_1)$ to $(x,-u,-z_1)$.  The fiber product $V_2 = V \times_X V$ is therefore $V \times \ZZ/2\ZZ$, and the \'etale equivalence relation $V_2 \double V$ has quotient $X$.  Now the skeleton $\oSigma(V_2)$ is $\osigma \times \ZZ/2\ZZ$, and the equivalence relation $\oSigma(V_2) \double \oSigma(V)$ identifies $\sigma$ with itself by the reflection that
  switches the two coordinates, since the toric divisors $\{x=z_1\}$ and $\{x=-z_1\}$ are interchanged by the involution.  It follows that $\oSigma(X)$ is the quotient $\osigma/(\ZZ/2\ZZ)$.  
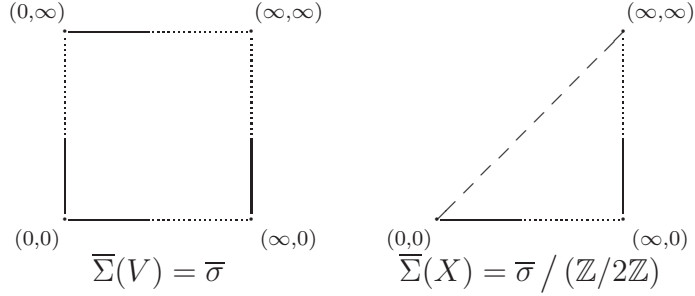
$$
\xymatrix@=0.15pc{*{\cdot}\ar@{-}[rrrr]^(-.3){(0,\infty)}\ar@{.}[rrrrrrrr] &&&&&&&&*{\cdot}\ar@{}[r]^(1.6){(\infty,\infty)}\ar@{.}[dddddddd]&&&&&
\\ \\ \\ \\ &&&&&&&&\ar@{-}[dddd]\\ \\ \\ \\ 
*{\cdot}\ar@{-}[uuuu]\ar@{.}[uuuuuuuu] 
\ar@{-}[rrrr]_(-.3){(0,0)}\ar@{.}[rrrrrrrr] &&&&&&&&
*{\cdot}\ar@{}[r]_(1.6){(\infty,0)}&
\\
\ar@{}[rrrrrrrr]_{\textstyle \oSigma(V)=\osigma}&&&&&&&&
}
\xymatrix@=0.15pc{&&&&&&&&*{\cdot}\ar@{.}[dddddddd]\ar@{{-}{--}{-}}[ddddddddllllllll]\ar@{}[r]^(1.6){(\infty,\infty)}&&&&&
\\ \\ \\ \\ &&&&&&&&\ar@{-}[dddd]\\ \\ \\ \\ 
*{\cdot} 
\ar@{-}[rrrr]_(-.3){(0,0)}\ar@{.}[rrrrrrrr] &&&&&&&&
*{\cdot}\ar@{}[r]_(1.6){(\infty,0)}&
\\
\ar@{}[rrrrrrrr]_{\textstyle \oSigma(X)=\osigma\,\big/\,(\ZZ/2\ZZ)}&&&&&&&&
}$$
The image of $\sigma^\circ$ is not homeomorphic to the relative interior of any cone, but the cones of the barycentric subdivision of $\Sigma(V)$ map isomorphically to their images, giving $\Sigma(X)$ the structure of a cone complex, as in Figure \ref{Fig:barycentric-quotient} of Section \ref{Sec:generalized}.
\end{example}

The functorial properties of retraction to the skeleton also carry over to toroidal stacks:
\begin{proposition}\label{Prop:functoriality} The construction of $\boldp_\cX$ is functorial: if $f:\cX 
  \to \cY$ is a sub-toroidal morphism of toroidal Deligne--Mumford
  stacks, then $f^\beth:X^\beth \to Y^\beth$ restricts to a map of generalized extended cone complexes $\oSigma(f): \oSigma(X) \to \oSigma(Y)$. 
    In particular,  $\boldp_\cY\circ f^\beth =\oSigma(f) \circ \boldp_\cX$, and if  $g:\cY \to \cZ$ is another  sub-toroidal  morphism, then $\oSigma(g)\circ\oSigma(f) = \oSigma(g\circ f)$.
   \end{proposition}

\begin{proof} Let $V \to \cY$ and  $W \to V \times_{\cY} \cX$  be \'etale coverings by schemes
such that the induced toroidal embeddings are without self-intersection. Then in the diagram 
$$\xymatrix@=.5pc{
W^\beth\ar[rr]\ar[dd]\ar[dr]&&X ^\beth \ar'[d][dd]\ar[dr]\\
 & V^\beth \ar[rr]\ar[dd]&& Y^\beth \ar[dd]\\
 \oSigma(W)\ar'[r][rr]\ar[dr]&&\oSigma(\cX) \ar@{.>}[dr]\\
&  \oSigma(V)\ar[rr]&&\oSigma(\cY) 
}$$ the dotted arrow exists since $f^\beth(\oSigma(\cX))$ lies in the image of $\oSigma(V)$.  It is a morphism of extended generalized cone complexes, since it is covered by the morphism of extended cone complexes $\oSigma(W) \to \oSigma(V)$. 
Now all but the right square are already known to be commutative, and the horizontal arrows are surjective, therefore the right square is commutative as well.
\end{proof}

 \begin{remark}
When constructing skeletons of toroidal De\-ligne--Mum\-ford stacks, it may be tempting to take colimits of diagrams of cone complexes in the category of topological stacks rather than in the category of topological spaces.  We avoid this for three reasons.  First, the cone complex $\Sigma(\cX)$ is a subset of $X^\an$ that lies over the generic point of $X$; if $v$ is in the cone $\sigma$ then the corresponding monomial valuation is finite on every nonzero function, by Formula~ (\ref{Eq:embedding-p}) in Section~\ref{sec:explicit-retraction}.  If $\cX$ does not have generic stabilizers, then no point of $\Sigma(\cX)$ has stabilizers, when considered as a point of $\cX^\an$.  Next, the same distinction between the colimit in the category of topological stacks and the colimit in the category of topological spaces appears already for toroidal embeddings of varieties with self-intersection, even when the underlying toroidal space has no nontrivial stack structure, as seen in Example~\ref{ex:Monodromy}.  Finally, the colimit of the diagram $\oSigma(V_2) \double \oSigma(V)$ in the category of topological stacks depends on the choice of \'etale cover, while the colimit in the category of topological spaces is independent of all choices.   The following example illustrates this possibility.
 
  \begin{example}
 Let $U \subset X$ be a toroidal scheme, so the embedding of $U \times \bG_m$ in $X \times \bG_m$ is also toroidal.  Fix an integer $n \geq 2$.  Then $X \times \bG_m$ has a natural \'etale cover $V$ induced by $z \mapsto z^n$ on $\bG_m$.  The resulting diagram $\oSigma(V_2) \double \oSigma(V)$ realizes $\oSigma(X)$ as the quotient of $\oSigma(V)$ by the trivial action of a cyclic group of order $n$.  In particular, the colimit in the category of topological spaces is $\oSigma(V)$, whereas the colimit in the category of topological stacks has a nontrivial stabilizer at every point that depends on the choice of $n$.
 \end{example}
 
 The issue seems to come from the fact that the underlying topological space of a scheme or a Berkovich space should come with a stack structure; possibly points should be replaced by the classifying stacks of appropriate Galois groups. This seems compatible with results of \cite{Chan-Melo-Viviani}. The simplicial space giving the \'etale topological type of $X^\beth$ and its restriction to $\oSigma(X)$ may provide an appropriate formalism.
  \end{remark}

\subsection{Monodromy of toroidal embeddings} \label{Sec:monodromy}
Let $U \subset \cX$ be a toroidal Deligne--Mumford stack.  For each \'etale morphism from a scheme $V \rightarrow \cX$, let $M_V$ be the group of Cartier divisors on $V$ that are supported on the boundary $V \smallsetminus U_V$, and let $S_V \subset M_V$ be the submonoid of effective divisors.  Let $M$ and $S$ be the \'etale sheaves associated to these presheaves, respectively.  In the language of logarithmic geometry, $S$ is the characteristic monoid sheaf associated to the open embedding $U \subset \cX$, and $M$ is the characteristic abelian sheaf.

\begin{proposition}
The sheaves $S$ and $M$ are locally constant  in the \'etale topology on each stratum $W \subset \cX$.
\end{proposition}

\begin{proof}
It suffices to check this for $M$, and it is enough to exhibit an \'etale cover on $W$ where the sheaf $M$ is constant. 
Since $\cX$ has an \'etale cover by a toroidal embedding of schemes without self-intersections, this follows from the fact that $M$ is constant on each stratum of any toroidal embedding without self-intersection \cite[Lemma~II.1.1, p.~60]{KKMS}.
\end{proof}

Fix a stratum $W \subset \cX$ and a geometric point $w$ of $W$.  The stalk $M_w$ is the group of \'etale local germs of Cartier divisors at $w$ supported on $\cX \smallsetminus U$,  and $S_w$ is the submonoid of germs of effective Cartier divisors.  Note that $M_w$ is a finitely generated free abelian group and $S_w$ is a sharp, saturated, and finitely generated submonoid that generates $M_w$ as a group.  Hence the dual cone $\sigma_w$, the additive submonoid of linear functions on $M_w$ that are nonnegative on $S_w$, is a strictly convex, full-dimensional, rational polyhedral cone in $\Hom(M_w, \RR)$.  Since $M_w$ is \'etale locally constant, there is a natural action of $\pi_1^{et}(W,w)$ on $M_w$ that preserves $S_w$.  See \cite{Noohi04} for details on \'etale fundamental groups of Deligne--Mumford stacks.

\begin{definition}\label{Def:monodromy} 
The monodromy group $H_{w}$ is the image of $\pi_1^{et}(W,w)$ in $\Aut(M_w)$.
\end{definition} 

\noindent The action of $\pi_1^{et}(W,w)$ on $M_w$ is determined by the induced permutations of the extremal rays of $\sigma_w$.  In particular, the monodromy group $H_w$ is finite.  

\begin{remark}
Note that any two geometric points $w$ and $w'$ in the same stratum $W \subset \cX$ have isomorphic monodromy groups, where the isomorphism is well-defined up to conjugation.  Similarly, the cones $\sigma_w$ and $\sigma_{w'}$ are isomorphic, by isomorphisms that are compatible with the actions of $H_w$ and $H_{w'}$, and well-defined up to conjugation by these actions.  In particular, the quotient $\sigma_w / H_w$ depends only on the stratum $W$, and not on the point $w$.
\end{remark}

To study the monodromy group $H_w$ at a point $w$ in a stratum $W \subset \cX$, we therefore study local charts given by \'etale covers by toroidal embeddings of schemes without self-intersection, where the monodromy is trivialized.

\begin{definition}
An \'etale morphism $V \rightarrow \cX$ from a scheme $V$ to a toroidal Deligne--Mumford stack $U \subset \cX$ is a \emph{small toric chart} around a geometric point $w$ of $\cX$ if
\begin{enumerate} 
\item the toroidal
embedding $U_V \subset V$ is without self-intersections,
\item there is a unique closed stratum
$\widetilde W \subset V$, and
\item the image of $\widetilde W$ contains $w$.
\end{enumerate}
\end{definition} 

Fix a small toric chart $V \rightarrow \cX$ and a point $\widetilde w$  of  $\widetilde W$ lifting $w$.  Since $V$ is without self-intersection, the \'etale sheaves $M$ and $S$ are constant on $\widetilde W$, so $\pi_1^{et}(\widetilde W, \widetilde w)$ acts trivially on $M_w$.  The skeleton $\oSigma(V)$ is simply the extended cone
 \[
 \osigma_V = \Hom(S_V, \RR_{\geq 0} \sqcup \{\infty\}).
 \] 

\begin{remark}
The monodromy group $H_w$ can be detected from a single small toric chart $V$ around $w$, as follows.  Let $V_2 = V \times_\cX V$.   Consider a point $y\in V_2$ lying over $x_1,x_2\in \widetilde W$, mapping to $w\in W$. Since $M$ is constant on $\widetilde W$, we can identify $M_{x_1}\simeq H^0(\widetilde W_V, M) \simeq M_{x_2}$. On the other hand pulling back we get $M_{x_2} \simeq  M_{y} \simeq M_{x_1}$. This determines an automorphism of $M_w$, and every element of $H_w$ occurs in this way. 
\end{remark}

We now state and prove the main technical result of this section, which says that the skeleton of an arbitrary toroidal embedding of Deligne--Mumford stacks decomposes as a disjoint union of extended open cones, one for each stratum, modulo the action of the respective monodromy groups.

\begin{proposition}\label{Prop:decomposition}
Let $W_1, \ldots, W_s$ be the strata of a toroidal Deligne--Mumford stack $\cX$, and let $w_i$ be a point in $W_i$.  Write $\sigma_i$ for the dual cone of $S_{w_i}$ and $H_i$ for the monodromy group at $w_i$. Let $\sigma_i^\circ$ be the relative interior of $\sigma_i$.  Then we have natural decompositions
$$\Sigma(\cX) \ \ = \ \ \sigma_1^\circ/ H_1 \ \sqcup \cdots  \sqcup\  \sigma_s^\circ/H_s,$$
and
$$\oSigma(\cX) \ \ = \ \ \osigma_1^\circ/ H_1\  \sqcup \cdots \sqcup\  \osigma_s^\circ/H_s.$$
Furthermore, if $V_i \rightarrow \cX$ is a small toric chart around $w_i$, with $V' = V_1 \sqcup \cdots \sqcup V_s$ and $V'_2 = V' \times_\cX V'$, then the natural map
\[
\varinjlim \left( \Sigma(V'_2) \double \Sigma(V') \right) \rightarrow \Sigma(\cX)
\]
is an isomorphism of generalized cone complexes, and 
\[
\varinjlim \left( \oSigma(V'_2) \double \oSigma(V') \right) \rightarrow \oSigma(\cX)
\] 
is an isomorphism of extended generalized cone complexes.
\end{proposition}

\begin{proof} 
First, note that it suffices to prove the statements for $\oSigma(\cX)$.  The decomposition statement for $\Sigma(\cX)$ follows from the one for $\oSigma(\cX)$, because $\sigma_i^\circ = \osigma_i^\circ \cap \Sigma(\cX)$.  Similarly, the isomorphism statement for $\Sigma(\cX)$ follows from the one for $\oSigma(\cX)$ because $\Sigma(V')$ and $\Sigma(V'_2)$ are the preimages of $\Sigma(\cX)$ in $\oSigma(V')$ and $\oSigma(V'_2)$, respectively.

Since each $V_i$ is a small toric chart, its skeleton is the single extended cone $\oSigma(V_i) = \osigma_i$, and hence
\[
\oSigma(V') = \osigma_1 \sqcup \cdots \sqcup \osigma_s.
\]
We write $\im(\osigma_i^\circ)$ for the image of $\osigma_i^\circ$ in $\oSigma(\cX)$.  First, we show that $\osigma_1^\circ \sqcup \cdots \sqcup \osigma_s^\circ$ surjects onto $\oSigma(\cX)$.  This does not follow from the definition of the skeleton, since $V_1 \sqcup \cdots \sqcup V_s$ need not surject onto $\cX$.  However, suppose $V^* \rightarrow \cX$ is a small toric chart around a point in $W_i$.  Then $V^*\times_\cX V_i$ contains a small toric chart, whose skeleton maps isomorphically by the two projections to $\osigma_{V^*}$ and $\osigma_i$, see \cite[Lemma 3.28 (2)]{Thuillier}.  Hence the natural map $\osigma_{V^*} \rightarrow \oSigma(\cX)$ factors through an isomorphism to $\osigma_i$.  Therefore, we can extend $V_1 \sqcup \cdots \sqcup V_s$ to a cover of $\cX$ by small toric charts and conclude that $\osigma_1 \sqcup \cdots \sqcup \osigma_s$ surjects onto $\oSigma(\cX)$.  Finally, each face of $\osigma_i$ corresponds to a stratum of $\cX$ whose closure contains $W_i$, so the image of each face of positive codimension in $\osigma_i$ is also in the image of a lower dimensional cone, and we conclude that $\osigma_1^\circ \sqcup \cdots \sqcup \osigma^\circ_s$ surjects onto $\oSigma(\cX)$.

Next, we observe that the images of $\osigma_1^\circ, \ldots, \osigma_s^\circ$ are disjoint, since a point of $\osigma_i^\circ$, considered as a point of $X^\beth$, extends to a point over $\Spec R$ whose reduction lies in $W_i$.  This shows $\oSigma(\cX) = \im(\osigma_1^\circ) \sqcup \cdots \sqcup \im(\osigma_s^\circ)$.

To prove the decomposition statement, it remains to show $\im(\osigma_i^\circ) = \osigma_i^\circ/ H_i$. Shrinking $V_i$ if necessary and writing $\widetilde W_i\subset V_i$ for the closed stratum, we may assume that the \'etale map of strata $\widetilde W_i \rightarrow W_i$ is finite onto its image.  Say $W'_i \subset W_i$ is the image of  $\widetilde W_i$.  Since each stratum $W_i$ is smooth, the fundamental group $\pi_1^{et}(W_i', w_i)$ surjects onto $\pi_1^{et}(W_i, w_i)$.  The sheaves $M$ and $S$ are trivial on $W_i'$, so every monodromy operator $g \in H_i$ is induced by some geometric point $y$ of $ W_i\times_\cX W_i$ over a pair of points $w$ and $w'$ in $W'_i$ that lie over $w_i$.  Let $V'_i$ be the component of $V_i \times_\cX V_i$ containing $y$.  Then the projections $\osigma_{V'_i} \double \osigma_i$ induce the identification $g: \osigma_i \xrightarrow{\sim} \osigma_i$  Therefore, two points in $\osigma_i^\circ$ that differ by an element of $H_i$ have the same image in $\oSigma(\cX)$.  Conversely, if $v$ and $v'$ are points in $\osigma_i^\circ$ that have the same image in $X^\beth$, then we can consider each as a point of $V_i^\beth$, and consider a point $y$ in $V_i^\beth \times_{\cX} V_i^\beth$ lying over $v$ and $v'$.  Then the monodromy operator associated to the reduction of $y$ maps $v$ to $v'$ in $\osigma_i^\circ$.  This proves the decomposition statement.

We now turn to the isomorphism statement.  As discussed in Section~\ref{Sec:complexes}, the category of generalized extended cone complexes is an extension of the category of extended cones in which any finite diagram of extended cones with face morphisms has a colimit.  Furthermore, the functor taking an extended cone to its underlying topological space extends to a faithful functor on generalized extended cone complexes that commutes with colimits.  We have seen that the natural map
\[
\varinjlim \left( \oSigma(V'_2) \double \oSigma(V') \right) \rightarrow \oSigma(\cX)
\] 
is surjective.  Let $\cX' \subset \cX$ be the image of $V'$.  Then
\[
\varinjlim \left( \oSigma(V'_2) \double \oSigma(V') \right) \rightarrow \oSigma(\cX')
\] 
is an isomorphism.  In particular, it is injective.  Composing with the inclusions $\oSigma(\cX') \subset {X'}^\beth$ and ${X'}^\beth \subset X^\beth$ shows that the map from the colimit to $\oSigma(\cX)$ is injective, and hence bijective.  Being a continuous bijection between compact Hausdorff spaces, it is a homeomorphism.  Finally, since all of the maps in the diagram are face maps, the natural integral structures are preserved, and the homeomorphism is an isomorphism of extended generalized cone complexes.
\end{proof}

\section{The skeleton of $\ocM_{g,n}$}
\label{Sec:proofs}

 In this section, we interpret the general construction of the retraction of a toroidal Deligne--Mumford stack onto its canonical skeleton in the special case of $\cM_{g,n} \subset \ocM_{g,n}$ and show that $\oSigma(\ocM_{g,n})$ is naturally identified with the tropical moduli space $\oM_{g,n}^\trop$.

\subsection{Versal deformation spaces} \label{Sec:versal} We begin by recalling some facts about deformations of stable curves \cite[Chapters XI, XII]{Arbarello-Cornalba-Griffiths}. Fix a  point $p$ in $\ocM_{g,n}$ corresponding to a stable curve $C$.  Then $p$ has an \'etale neighborhood $V_p \rightarrow \ocM_{g,n}$ in which the locus parametrizing deformations of $C$ in which the node $q_i$ persists is a smooth and irreducible principal divisor $D_i$ with defining equation $f_i$, and the collection of divisors corresponding to all nodes of $C$ has simple normal crossings.  Shrinking $V_p$ if necessary, we may assume that the locus in $V_p$ parametrizing singular curves is the union of these divisors and, for each collection of nodes $\{q_i\}_{i \in I}$, the corresponding intersection
\[
W_I = \bigcap_{i \in I} D_i
\]
is irreducible.  The completion of $V_p$ at $p$ is a formal affine scheme and the $f_i$ are a subset of a system of formal local coordinates.  Furthermore, the dual graph of any curve in the family parametrized by $V_p$ is a contraction of the dual graph $\WG$ of $C$.

The curves parametrized by $D_i$ are exactly those having dual graphs in which the edge $e_i$ corresponding to the node $q_i$ is not contracted.  More generally, the locally closed stratum
\[
W_I^\circ \subset W_I,
\] 
consisting of points that are in $D_i$ if and only if $i \in I$, parametrizes those curves whose dual graph is $\WG_{/E'}$, the graph in which the edges in $E'$ are contracted and only the edges $\{e_i\}_{i \in I}$ remain, where $E' = \{e_j\}_{j \notin I}$.  

Since the defining equation $f_i$ of $D_i$ on $V_p$ measures deformations of the node $q_i\in C$, it has the following interpretation in terms of the local defining equations of the curve at the node.  Consider a valuation ring $R$ and a morphism $\phi: \Spec R\to V_p$, corresponding to a curve $C_R$ over $\Spec R$.  Assume the closed point in $\Spec R$ maps into the stratum $W_I^\circ$. Then, for $i \in I$, the node $q_i$ in the special fiber of $C_R$ has an \'etale neighborhood in $C_R$ with defining equation $xy = f_i$, where we identify $f_i$ with its image in $R$.

\subsection{Monodromy on $\oSigma(\ocM_{g,n})$} 

Let $V \rightarrow \ocM_{g,n}$ be a small toric chart around a point $p$ in the stratum $\cM_\WG$, such as the versal deformation spaces discussed above.  Then the skeleton $\oSigma(V)$ is a single copy of the extended cone $\osigma_\WG$.  By Proposition~\ref{Prop:decomposition}, the image of $\osigma^\circ_\WG$ in $\oSigma(\ocM_{g,n})$ is the quotient of $\osigma_\WG^\circ$ by the monodromy group $H_\WG$.  Recall that, by definition, $H_\WG$ is the image of $\pi_1^{et}(\cM_\WG,p)$ in $\Aut(\osigma_\WG)$.

\begin{proposition} \label{prop:full-monodromy}
The monodromy group $H_\WG$ is the image of $\Aut(\WG)$ in the set of permutations of $E(\WG)$.
\end{proposition}

\begin{proof}
To compute the monodromy group $H_{\WG}$, we consider the Galois cover $\widetilde \cM_\WG \rightarrow \cM_\WG$, with Galois group $\Aut(\WG)$, from Section~\ref{sec:autg-cover}.  The pullbacks of the sheaves $M$ and $S$ are trivial on $\widetilde M_\WG$ because, by construction, the cover $\widetilde \cM_\WG \rightarrow \cM_\WG$ trivializes the locally constant sheaves of sets on $\cM_\WG$ whose stalk at a point $x$ is the set of nodes of the corresponding curve $C_x$.  By the discussion of versal deformations above, these sets form a group basis for $M$ and a monoid basis for $S$.  The action of $\pi_1^{et}(\cM_\WG, p)$ therefore factors through its quotient $\Aut(\WG)$, acting in the natural way on $\osigma_\WG$.
\end{proof}

\begin{corollary}
The skeleton $\oSigma(\ocM_{g,n})$ decomposes as a disjoint union
\[
\oSigma(\ocM_{g,n}) = \bigsqcup_\WG \  \osigma^\circ_\WG / \Aut(\WG).
\]
\end{corollary}

\subsection{Proof of Theorem~\ref{Th:functor}}
We have seen that both the skeleton $\oSigma(\ocM_{g,n})$ and the tropical moduli space $\oM_{g,n}^\trop$ decompose naturally as disjoint unions over isomorphism classes of stable graphs of genus $g$ with $n$ legs
\[
\oSigma (\ocM_{g,n}) = \bigsqcup_\WG \ \osigma^\circ_\WG / \Aut(\WG) = \oM_{g,n}^\trop.
\]
We now show that these bijections induce an isomorphism of extended generalized cone complexes and are compatible with the naive set theoretic tropicalization map from Definition~\ref{Def:Trop}.

Choose a small toric chart $V_\WG \rightarrow  \ocM_{g,n}$ around a point in each stratum $\cM_\WG$.  Let
\[
V = \bigsqcup_\WG V_\WG,
\]
with its \'etale map $V \rightarrow \ocM_{g,n}$.  Then $\oSigma(V) = \bigsqcup_\WG \osigma_\WG$ and, by Proposition~\ref{Prop:decomposition}, the skeleton $\oSigma(\ocM_{g,n})$ is naturally identified with the colimit of the diagram $\oSigma(V_2) \double \oSigma(V)$, where $V_2 = V \times_{\ocM_{g,n}} V$.  By Proposition~\ref{Prop:generalized-decomposition}, we can replace this diagram with one in which each cone $\osigma_{\WG}$ appears exactly once.  By Proposition~\ref{prop:full-monodromy}, the self-maps $\osigma_\WG \rightarrow \osigma_\WG$ in this diagram are exactly those induced by an automorphism of $\WG$.  Furthermore, by the discussion of versal deformations in Section \ref{Sec:versal}, the closure in $\ocM_{g,n}$ of any stratum corresponding to a contraction of $\WG$ contains $\cM_{\WG}$, so the proper inclusions of faces $\jmath: \osigma_{\WG'} \rightarrow \osigma_\WG$ in this diagram are exactly those corresponding to graph contractions $\varpi: \WG \rightarrow \WG'$.

The same diagram of extended cones is considered in Section~\ref{Sec:moduli-trop-glue}, where its colimit is identified with $\oM_{g,n}^\trop$, giving an isomorphism 
$\overline \Phi_{g,n}$ 
of extended generalized cone complexes, which restricts to an isomorphism of generalized cone complexes $\Phi_{g,n} : \Sigma(\ocM_{g,n}) \rightarrow M_{g,n}^\trop$, as required.

It remains to check that this identification agrees with the naive set theoretic tropicalization map.  
Suppose $C = C_p$ is a curve over a valued field $K$ that extends to a curve $C_R$ over $\Spec R$, and let $\WG$ be the dual graph of the special fiber.  Then the point $p$ has an \'etale neighborhood in $\ocM_{g,n}$ in which each node $q_i$ of the reduction of $C$ is defined by an equation $xy = f_i$, with $f_i$ in $R$.  The naive set theoretic tropicalization map takes $p$ to the metric graph with underlying weighted graph $\WG$ in which the length of the edge $e_i$ corresponding to $q_i$ is $\val_C(f_i)$; see Definition~\ref{Def:Trop}.
By the discussion of versal deformations in Section \ref{Sec:versal}, the divisors  $D_i=(f_i)$  also give a basis for the monoid $S_p$.  The explicit description of the retraction to the skeleton, from Section~\ref{sec:explicit-retraction}, then shows that this retraction takes $p$ to the same metric graph, and the theorem follows.\qed 
 
\begin{remark}
We note that our proof of Part (1) of Theorem \ref{Th:functor} is based on the observation that $\oSigma(\ocM_{g,n})$ and $\oM_{g,n}^\trop$ are put together in the same way from the same extended cones, and does not require the analytic interpretation of $\oSigma(\ocM_{g,n})$ as a skeleton. \end{remark}

\section{Tropical tautological maps}
\label{Sec:tautological}

\subsection{Curves and tropical curves: the analogy of strata}

As discussed in Section~\ref{Sec:strata}, the strata in $\ocM_{g,n}$ correspond to stable graphs $\WG$, and the codimension of the stratum $\cM_\WG$ is the number of edges in $\WG$.  Furthermore, $\cM_\WG$ is contained in $\overline{\cM_{\WG'}}$ if and only if there is a graph contraction $\WG \rightarrow \WG'$.

The natural stratification of the tropical moduli space $M_{g,n}^\trop$ is similar, but the inclusions are reversed, as seen in Section~\ref{Sec:moduli-trop-glue}.  The stratum $M_{\WG}^\trop$ parametrizing stable tropical curves with underlying graph $\WG$ is contained in $\overline {M_{\WG'}^\trop}$ if and only if there is a graph contraction $\WG' \rightarrow \WG$.  As the orders are reversed, dimension and codimension are also interchanged; the dimension of $M_{\WG}^\trop$ is equal to the number of edges in $\WG$;
see \cite[Thm. 4.7]{C2}.  

This order reversing correspondence between stratifications may be seen as a consequence of Theorem~\ref{Th:functor}.  The tropical moduli space is the finite part of the skeleton of the moduli space of curves, and there is a natural order reversing correspondence between strata in a toroidal space and cones in the associated complex.  See Remark~\ref{rem:strata}.

\subsection{Tropical   forgetful maps and their sections}\label{Sec:forgetful}

Assume as usual $2g-2+n>0$. In the algebraic situation there is a natural {\it forgetful} morphism 
$$
\pi=\pi_{g,n}:\ocM_{g,n+1} \longrightarrow \ocM_{g,n}
$$ obtained functorially by forgetting the last marked point and replacing the curve by its stabilization,   if necessary. It was shown by Knudsen that this exhibits $\ocM_{g,n+1}$ as the universal curve over $\ocM_{g,n}$. On the level of coarse moduli spaces, we have that the fiber of $\oM_{g,n+1} \to \oM_{g,n}$ over the point $[(C;p_1,\ldots,p_n)]$ is the quotient  $C/\Aut(C;p_1,\ldots,p_n)$.  

The forgetful map $\pi_{g,n}$ has $n$ {\it tautological sections}, $\sigma_1,\ldots, \sigma_n$, corresponding to the  marked points.
 Knudsen identified the  image of $\sigma_i$   as the locus in $\ocM_{g,n+1}$  where the
 marked points
 $p_i$ and $p_{n+1}$ lie on a smooth rational component meeting the rest of the curve in a unique point,  and containing no other marked point. 

Let us construct a natural forgetful map  in the tropical setting.
$$\pi^\trop_{g,n}=\pi^\trop:
\oM_{g,n+1}^\trop \longrightarrow \oM_{g,n}^\trop.
$$
  Given a tropical curve $\TC\in \oM_{g,n+1}^\trop $, denote by   $v$  the vertex where the $(n+1)$st leg of $\TC$ is attached;
let us remove this leg
and  denote by
$\TC_*$ the resulting tropical curve. If  $\TC_*$ is not stable, then
  $h(v)=0$ and the valence of $v$ in $\TC_*$ is 2;
it is clear that  one of the two following cases occurs. 
Case (1): 
adjacent to $v$ there are  a leg $l$ and
 an edge $e_1$, whose second endpoint we denote by $v_1$. 
Case (2): adjacent to $v$ there are
  two edges $e_1,e_2$, the second    endpoint of which we denote by $v_1$ and $v_2$ respectively. In these cases we replace $\TC_*$ by a  stable tropical curve, $\widehat{\TC_*}$, as follows. 
 
In  case (1) we remove $e_1$ and $v$, and reattach a leg $l'$ at $v_1$, as in the following picture.

$$
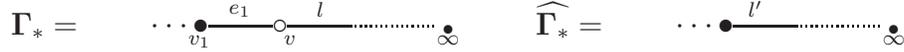
\xymatrix{
\TC_*=& *{\cdots\bullet}\ar@{-}[r]^(0.6){e_1}_(0,2){v_1}_(1,1){v }&*{\circ}
 \ar@{-}[r]^(0.5){l}\ar@{.}[rr]_(1){\infty}&&{\scriptstyle{\mathop{\bullet}}} 
&\widehat{\TC_*}=&*{\cdots\bullet} \ar@{-}[r]^(0.5){l'}\ar@{.}[rr]_(1){\infty}&&{\scriptstyle{\mathop{\bullet}}} 
  }$$

 In case (2) we define  $\widehat{\TC_*}$ to be the graph obtained by removing $v, e_1,e_2$ from $\TC_*$, and adding an edge $e'$ with endpoints $v_1,v_2$ and length 
equal to
 $\ell(e_1)+\ell(e_2)$:
 
 $$
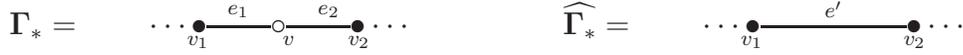
\xymatrix{
 \TC_*=&*{\cdots\bullet}\ar@{-}[r]^(0.6){e_1}_(0,2){v_1}_(1,1){v }& *{\circ}\ar@{-}[r]^(0.5){e_2}_(0.8){v_2}& *{\bullet\cdots}
 &&\widehat{\TC_*}=&*{\cdots\bullet}\ar@{-}[rr]_(0,1){v_1}_(0.9){v_2 }^(0.5){e'}&& *{\bullet\cdots}
}$$
 
Notice that   in both cases 
 we have a canonical point (not a vertex), $p_v$, on $\widehat{\TC_*}$ corresponding to $v$.
 Indeed, in case (1), if the length of $e_1$ is finite,  $p_v$ is the point on $l'$  at distance $\ell(e_1)$  from $v_1$. If $\ell(e_1)$ is infinite, then $p_v$ is   the infinity point on the new leg $l'$.

 In case (2), 
  if the edge  $e_1$ (say)  has finite length, then
 the point $p_v$ is the point of $e'$ at distance $\ell(e_1)$  from $v_1$.
 If   $\ell(e_1)=\ell(e_2)=\infty$, then $p_v$ is defined to be the infinity point on the new edge $e'$.

We thus obtain a continuous cellular map,
 $ \pi^\trop:
\oM_{g,n+1}^\trop \to \oM_{g,n}^\trop
$   sending $\TC$ to $\TC_*$ if stable, and to $\widehat{\TC_*}$ otherwise. 

Let us now show that the tropical forgetful map realizes $\oM_{g,n+1}^\trop$ as the coarse universal curve over $\oM_{g,n}^\trop$.
\begin{proposition} 
\label{prop:uni}
Let $\TC\in \oM_{g,n}^\trop$ and let $F_\TC$ be 
 the fiber  of $\pi^\trop:\oM_{g,n+1}^\trop \to \oM_{g,n}^\trop$ over   $\TC$. 
 Then $F_\TC$ is homeomorphic to $\TC/\Aut(\TC)$. 
 Moreover, if $\TC\in M^\trop_{g,n}$ then $F_\TC$
 is isometric  to $\TC/\Aut(\TC)$.
\end{proposition}

\begin{proof} 
We have a map $F_\TC \to \TC/\Aut(\TC)$ by sending   a tropical curve with $n+1$ legs 
(the last of which adjacent to the vertex $v$) to the  point $p_v$ corresponding to $v$.  To obtain the inverse, we identify $\TC/\Aut(\TC)$ to the space of points $p\in \TC$   up to isometries preserving the weights on the vertices.
Then by attaching a leg at $p$ we obtain a tropical curve, $\TC_p$, with $n+1$ marked points.
More precisely, we have the following possibilities.
If $p$ is a vertex of $\TC$ then we simply add a leg adjacent to $p$.
If $p$ is not a vertex of $\TC$, then we declare $p$ to be  a vertex of weight zero, and attach 
a leg at it; the new vertex $p$ has thus valency 3, and hence the tropical curve $\TC_p$   is stable.

It is clear that as $p$ varies in its $\Aut(\TC)$-orbit, the isomorphism class of $\TC_p$   does not change. So the above construction descends
to a map $ \TC/\Aut(\TC) \to F_\TC$, which is
 the inverse of the map defined before.
 
It is clear that this map is a homeomorphism, and an isometry if all edges of $\TC$ have finite length.
\end{proof}
\begin{remark} It would be interesting to develop the theory on a stack level in such a way that the fiber is exactly the curve.
\end{remark}

\begin{example}
Consider $\oM_{1,1}^\trop$. It has two strata, one of dimension zero and one of dimension one.
The dimension zero stratum corresponds to the (unique) curve $\TC_0$  with one vertex $v$ of weight 1, no edges,
and a leg attached to $v$.
Thus $\Aut(\TC_0)=0$.   

\begin{figure}[ht]
\begin{equation*}
\xymatrix@=.5pc{
&&&&&&&&&&&&&&&&&&&&&&&&&&&&&&&&&&&\\
&&&&&&&&&&&&&&&&&&&&&&&&&&&&&&&&&&&\\
 F_{{\TC_0}}:&&*{\bullet}\ar@{{-}{.}}[ruu]^(0.8){l_2}\ar@{-}[r]_(0,0001){v}\ar@{.}[rrr]_(0.5){l_1}&&&&&&
 *{\bullet}\ar@{-}[rrr]_(0,0001){v}^(0.5){\ell}&&&*{\circ}\ar@{{-}{.}}[ruu]^(0.8){l_2}\ar@{-}[r]\ar@{.}[rrr]_(0.5){l_1}
 &&&&&
  *{\bullet}\ar@{-}[rr]_(0,0001){v}&\ar@{.}[rr]^(0.5){\infty}&\ar@{-}[rr]&&*{\circ}\ar@{{-}{.}}[ruu]^(0.8){l_2}\ar@{-}[r]\ar@{.}[rrr]_(0.5){l_1}
&&&&&&&&&&&&&&&&&&&&&&\\
&&\ar@{|->}[dd] &&&&&&&\ar@{|->}[ddll]&&&&&&&&&&&\ar@{|->}[ddrr]\\   
 &&&&&&&&&&&&&&&&&&&&&&&&&&&&&&&&&&&&&&&&&&&& \\
&&&&&&&&&&&&&&&&&&&&&&&&&&&&&&&&&&&&&&&&&&&&
\\
{\TC_0}=&&*{\bullet}\ar@{-}[rrrrrrrrrrrrrrrrrr]_(0,0001){v}_(0.22){p}
^(0,1){\ell}&&&&*{\scriptstyle{\bullet}}&&&&&&&&&&&&&\ar@{.}[rrrr]_(1){\infty}&&&&*{\scriptstyle{\bullet}}
}
\end{equation*}
\caption{Fiber of $\pi^{\trop}$ over the smallest stratum of $\oM_{1,1}^\trop$}\label{Fig:fiber-small}
\end{figure}
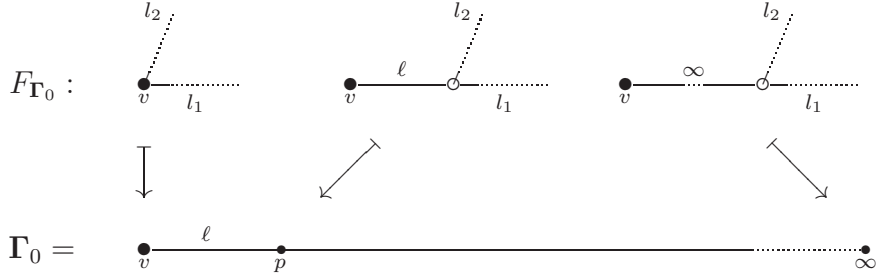
Figure \ref{Fig:fiber-small} represents 
$F_{\TC_0}\subset \oM_{1,2}^\trop$ and its isometry with $\TC_0$;
at the top we have   the three types of curves parametrized by $F_{\TC_0}$,
and, at the bottom, the corresponding point   of $\TC_0$. Notice that the curves on the right and on the left are unique, whereas  in the middle they vary  with $\ell \in \RR_{>0}$.
 The one-dimensional stratum of $\oM_{1,1}^\trop$ is a copy of $\RR_{>0}\sqcup \{\infty\}$; it parametrizes
curves $\TC_d$ whose   graph has one vertex of weight zero, one loop-edge of length      $ d\in\RR_{>0}\cup \{\infty\}$, and one leg. Then  $\Aut(\TC_d)=\ZZ/2\ZZ$
where the involution  corresponds to switching the orientation on the loop-edge.
The quotient $\TC_d/\Aut(\TC_d)$ is drawn below

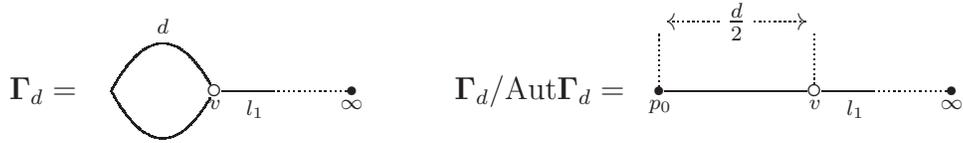
\begin{equation*}
\xymatrix@=.5pc{
&&&&&&&&&&&&\ar@{.}[d]\ar@{<.}[rr]&&{\frac{d}{2}}\ar@{.>}[rr]&&\ar@{.}[d]&&&&\\
\TC_d=& *{} \ar@{-}@/^1.5pc/[rrr]^(0.5){d}\ar@{-}@/_1.5pc/[rrr]&&&*{\circ}\ar@{-}[rr]_(0.01){v}_(0.6){l_1}&\ar@{.}[rrr]_(1){\infty}&&&*{\scriptstyle{\bullet}}&&&
\TC_d/\Aut\TC_d=&
*{\scriptstyle{\bullet}}\ar@{-}[rrrr]_(0.01){p_0}&&&&*{\circ}\ar@{-}[rr]_(0.01){v}_(0.6){l_1}&\ar@{.}[rrr]_(1){\infty}&&&*{\scriptstyle{\bullet}}&&&\\ &\\
}
\end{equation*}

The following Figure \ref{Fig:fiber} represents at the top one-dimensional families of curves of $F_{\TC_d}$;
the curves  on the left  vary with $0<\ell<\frac{d}{2}$, while on the right
with $\ell'\in \RR_{>0}$.  
The middle row represents the three remaining points of $F_{\TC_d}$.
\begin{figure}[ht]
\begin{equation*}
\xymatrix@=.5pc{
&&&&&&&&&&&&&&&&&&&&&&&&&&&&&&&&&&&&&&&&&&&&&&\\
&&&&&*{\circ}\ar@{{-}{.}}[ll]_{l_1}\ar@{-}[rrr]_{\ell}  \ar@{-}@/^.9pc/[rrr]^{d-\ell}
&&&*{\circ}\ar@{{-}{.}}[rr]^{l_2}&&&&&&&
*{\circ}\ar@{-}@(ul,dl)_{d}\ar@{-}[rrr]^{\ell'}&&&*{\circ}\ar@{{-}{.}}[rru]^(0.8){l_1}\ar@{{-}{.}}[rrd]^(0.7){l_2}&&&&&&&&\\
&&&&&&&\ar@{|->}@/_.5pc/[dddddddr]&&&&&&&&\ar@{|->}@/_.5pc/[rddddddd]&&&&&&&&&&&&&&&&&&&&&&&&&&&&&&&&&&&\\
&&&&&&&&&&&&&&&&&&&&&&&&&&&&&&&&&&&&&&&&&&&&&&&\\
&&&&&&&&&&&&&&&&&&&&&&&&&&&&&&&&&&&&&&&&&&&&&&\\
&*{\circ}\ar@{{-}{.}}[l]_{l_1}\ar@{-}@/_.9pc/[rr]^{\frac{d}{2}}  \ar@{-} @/^.9pc/[rr]^{\frac{d}{2}}
&&*{\circ}\ar@{{-}{.}}[rr]^{l_2}&&&&&&&&
*{\circ}\ar@{-}@(ul,dl)_{d}\ar@{{-}{.}}[rru]^(0.8){l_1}\ar@{{-}{.}}[rrd]^(0.7){l_2}&&&&&&&
*{\circ}\ar@{-}@(ul,dl)_{d}\ar@{-}[r]\ar@{.}[rr]^{\infty}&\ar@{-}[r]&*{\circ}\ar@{{-}{.}}[rru]^(0.8){l_1}\ar@{{-}{.}}[rrd]^(0.7){l_2}&&&\\
&&&&&&&&&&&&&&&&&&&&&&&&&&\\
&&\ar@{|->}[dd] &&&&&&&&&&\ar@{|->}[dd]&&&&&&&&&\ar@{|->}[dd]\\   
 &&&&&&&&&&&&&&&&&&&&&&&&&&&&&&&&&&&&&&&&&&&& \\
&&&&&&&&&&&&&&&&&&&&&&&&&&&&&&&&&&&&&&&&&&&&\\
{\bf \TC_d/\Aut{\TC_d}}=&&*{\scriptstyle{\bullet}}\ar@{-}[rrrrrrrrrr]_(0,04){p_0}_(0.7){p}
^(0,85){\ell}&&&&&&&*{\scriptstyle{\bullet}}&&&*{\circ} \ar@{-}[rrrrrrr]_(0.7){p'}_(0.04){v}^(0.4){\ell'}&&&&&*{\scriptstyle{\bullet}}\ar@{.}[rrrr]_(1){\infty}&&&&*{\scriptstyle{\bullet}}\\
&&\ar@{.}[u]\ar@{<.}[rrrrr]&&&&&\frac{d}{2}\ar@{.>}[rrrrr]&&&&&\ar@{.}[u]&&&&&&&&&&&&&&&&&&&&&\\
&&&&&&&&&&&&&&&&&&&&&\\
}
\end{equation*}
\caption{Fiber of $\pi^{\trop}$ over   $[\TC_d]\in \oM_{1.1}^\trop$ with $d>0$.}\label{Fig:fiber}
\end{figure}

\

Finally, Figure \ref{Fig:m12bar} depicts the forgetful map from $\oM_{1,2}^\trop$ to $\oM_{1,1}^\trop$.

\begin{figure}[ht]
\begin{equation*}
\xymatrix@=0.5pc
{
&&{\scriptstyle{\bullet}}&&&&&{\scriptstyle{\bullet}}&&&&&&&&&&{\scriptstyle{\bullet}}\\
&&&&&&&&&&&&&&&&&\\
&&\ar@{.}[uu]&&&&&\ar@{.}[uu]&&&&&&&&&&\ar@{.}[uu]\\
&\oM_{1,2}^\trop=&&&&&&&&&&&&&&&&\\
&&*{\scriptstyle{\bullet}}\ar@{-}[uuu]_(0.2){{\TC_0}}\ar@{--}[rrrrrdd]\ar@{-}[rrrrrrrrrrr]&&
&&& *{\scriptscriptstyle{\bullet}}\ar@{-}[dd]^{\frac{d}{2}}\ar@{-}[uuu]_(0.2){{\TC_d/\Aut\TC_d}}&&&&&\ar@{.}[rrrrr]*{\scriptstyle{\bullet}}&&&&&\ar@{-}[uuu]_(0.2){{\TC_\infty/\Aut\TC_\infty}}&&&&\\
&&&&&&&&&&&&&&&&&&&&&&&&&\\
&&&&&&&*{\scriptstyle{\bullet}}\ar@{--}[rrrrrdd]&&&&&&&&&&&&\\ 
&&&&&&&\ar[ddd]^{\pi^\trop_{1,1}}&&&&&&&&&&\ar@{-}[uuuu]\\
&&&&&&&&&&&&\ar@{.}[rrrrrdd]&&&&&\\
&&&&&&&&&&&&&&&&&&&&&&\\
& &&&&&&&&&&&&&&&& *{\scriptstyle{\bullet}}\ar@{.}[uuuu]\\
&\oM_{1,1}^\trop=&*{\scriptstyle{\bullet}}\ar@{-}[rrrrr]_(0){\TC_0} &&&&& *{\scriptstyle{\bullet}}\ar@{-} [rrrrrr]_(0){\TC_d} &&&\ar@{.} [rrrrrrr] _(1){\TC_\infty} &&&&&&&*{\scriptstyle{\bullet}}\\
\\
}
\end{equation*}
\caption{The forgetful map  $\pi^{\trop}_{1,1}$.}\label{Fig:m12bar}
\end{figure}
\end{example}
 
 \noindent{\it {Proof of the commutativity of the first diagram of Theorem~\ref{Th:tautological}.}}
Having defined the map $\pi^\trop$   we can consider the diagram

$$\xymatrix{
\ocM_{g,n+1}^\an\ar[rr]^{\Trop}\ar[d]_{\pi^\an} &&
 \oM_{g,n+1}^{\trop}\ar[d]^{\pi^\trop} \\
\ocM_{g,n}^\an\ar[rr]^{\Trop} && \oM_{g,n}^{\trop}
} $$ 
where $\pi^\an$ is the morphism canonically associated to the algebraic forgetful map
$\pi:\ocM_{g,n+1} \longrightarrow \ocM_{g,n}$ (explicitly described below).
 Let $[C]$ be a geometric point of $\oM_{g,n+1}^\an$,
so that $[C]$ is represented by a pair 
$$
(\val_C:K\longrightarrow \RR\sqcup \{\infty\}, \  \  \mu_C:\Spec R \longrightarrow \oM_{g,n+1}),
$$
where $K \supset k$ is an algebraically closed extension field, $\val_C$ is a valuation on $K$ that extends the trivial valuation on $k$,
and $R\subset K$ is the valuation ring. 
The morphism $\mu_C$    corresponds to a family of stable curves, $C\to \Spec R$; we write $C_s$ and $C_K$ for its special and generic fiber.
Now set 
$[C']:=\pi^\an([C])\in \oM_{g,n}^\an$; this point is represented by the pair
$$
(\val_C:K\longrightarrow \RR\sqcup \{\infty\}, \  \  \pi\circ\mu_C:\Spec R \longrightarrow \oM_{g,n}).
$$
It is clear that  the special fiber $C'_s$ of $C'\to \Spec R$ is equal to $\pi(C_s)$.

Denote by $\WG$ the dual graph of $C_s$.
Recall   that we have $\Trop([C])=(\WG,\ell_C)$ where the length function $\ell_C$ is determined by the valuation $\val_C$, and by the local geometry of the family $C\to \Spec R$ at the nodes of its special fiber, $C_s$; see Definition~\ref{Def:Trop}.
Similarly, writing    $\WG'=\WG_{C_s'}$, we have
$$
\Trop(\pi^\an([C]))=\Trop([C'])=(\WG',\ell_{C'}). 
$$
Now, by our description of the map $\pi^\trop$, it is clear that the graph underlying $\pi^\trop(\Trop([C]))=\pi^\trop(\WG,\ell_C)$  is equal to 
the dual graph of the algebraic curve $\pi(C_s)$; on the other hand $\pi(C_s)=C'_s$.
We conclude that the graph underlying  $\Trop(\pi^\an([C]))$ and $\pi^\trop(\Trop([C]))$
is the same. It remains to prove that the length functions of these two points are the same.
Let us write
$$
 \pi^\trop(\Trop([C]))= \pi^\trop(\WG,\ell_C)=(\WG',\widetilde{\ell}\ ) 
$$   
where $\widetilde{\ell}$ is   determined by $\pi^\trop$, as explained before the statement of Proposition~\ref{prop:uni}.  
To show that $\ell_{C'}=\widetilde{\ell}$,   notice that
they depend on the same 
 valuation,
 namely $\val_C$; hence we have to analyze   the total spaces of the families locally at the nodes of their special fibers.

If the curve $C_s$ remains stable after removing its (n+1)-st marked point, then 
the total space of the family $C'\to \Spec R$ (regardless of its marked points) is exactly the same as that of $C\to \Spec R$, 
and the dual graph of $C'_s$  is obtained from $\WG$ by removing one leg;
so the edges are the same and 
  $\ell_{C'}$ and $\widetilde{\ell}$ are both equal to $\ell_C$.

Now suppose $C_s$ is not stable after the removal of its last marked point.
The situation is identical to the one we had in the tropical setting, when defining the map $\pi^\trop$;   as on that occasion, we now distinguish  two cases.
In case (1) the removal of the last marked point from $C_s$ creates a ``one-pointed rational tail", i.e. a
smooth rational component, $E$, attached to the rest of $C_s$ at only one node, and having only one marked point on it.  In the family $C'\to \Spec R$ the component $E$ is contracted to a smooth point of $C'_s$, and the local geometry of $C'$ near the rest of $C'_s$ is the same.
So, the graph $\WG'$ has one fewer edge than $\WG$ and both 
  $\ell_{C'}$ and $\widetilde{\ell}$ coincide with the restriction of $\ell_C$ to the edges of  $\WG'$. 
  
  The remaining case (2) is more interesting. Here the removal of the last marked point
  creates an ``unpointed exceptional component", i.e. a smooth rational component, $E$,
  with no marked points and such that 
  $$
  E\cap \overline{C_s\smallsetminus E}=\{\sp_1,\sp_2\},
$$
with $\sp_1$ and $\sp_2$ nodes of $C_s$. Let $xy=f_i$ be the local equation of $C$ at
$\sp_i$. Then, denoting by $e_i\in E(\WG)$ the edge corresponding to $\sp_i$, we have
\begin{equation}
\label{le12}
\ell_C(e_i)=\val_C(f_i),\quad \quad i=1,2.
\end{equation}
The curve $C'_s$ is obtained from $C_s$ by collapsing $E$ to a node;
its dual graph is obtained from the dual graph of $C_s$ by 
removing the last leg (adjacent to the vertex, $v$, corresponding to $E$),
removing $v$, and ``merging" $e_1$ and $e_2$ into a unique edge $e'$. 
Now, the total space of $C'$,   locally at the node of $C'_s$ corresponding to $e'$,
has equation $xy=f_1f_2$ and hence
$$
\ell_{C'}(e')=\val_C(f_1f_2)=\val_C(f_1)+\val_C(f_2).
$$
On the other hand, by definition of $\pi^\trop$, we have
$$\widetilde{\ell}(e')=\ell_C(e_1)+\ell_C(e_2)=\val_C(f_1)+\val_C(f_2)
$$
by \eqref{le12}. Hence $\ell_{C'}(e')=\widetilde{\ell}(e')$. Of course, all the remaining edges of 
$\WG'$ are naturally identified with edges of $\WG$, and the values of
 $\ell_{C'}$ and $\widetilde{\ell}$ on them is equal to the value of $\ell_C$.
 The proof of the commutativity of the first diagram in Theorem~\ref{Th:tautological} is complete. \qed

\

We now proceed to define   the  tautological sections  of the forgetful maps,
in analogy with the algebraic case.
Let $\TC=(V,E,L,h,\ell)$ be a tropical curve in $\oM^{\trop}_{g,n}$.
For  $i\in \{1,\ldots,n\}$ we define the tropical curve 
$$\TC^i=(V^i,E^i,L^i,h^i,\ell^i)$$
as follows.
Let $l_i\in L$ be the $i$-th leg of $\TC$ and $v\in V$ its endpoint.
$\TC^i$ is obtained by attaching an edge $e_0$ at $v$
whose second endpoint we denote by  $v_0$.
We set
$V^i=V\cup \{v_0\}$ and $E^i=E\cup \{e_0\}$;   the weight function $h^i$ is the extension of
$h$ such that $h^i(v_0)=0$;   the length function $\ell^i$ is the extension of
$\ell$ such that $\ell^i(e_0)=\infty$.
Finally we remove the leg $l_i$  and attach two legs
at $v_0$, denoted by $ l'_i$ and $l_{n+1}$; summarizing
  $L^i=L\smallsetminus\{l_i\}\cup \{l'_i, l_{n+1}\}$. Here is an picture with $i=n=1$:
 
  
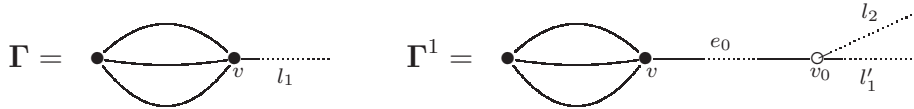
\begin{figure}[ht]
\begin{equation*}
\xymatrix@=.5pc{
\\
&&&&&&&&&&&&&&&&&&&&&&&&&&&&&&&&&&\\
 \TC=&*{\bullet}
  \ar @{-} @/_.2pc/[rrrr] \ar @{-} @/_1.5pc/[rrrr]\ar@{-} @/^1.1pc/[rrrr]&&&&*{\bullet}\ar@{-}[r]_(0,1){v}\ar@{.}[rrr]_(0.5){l_1}&&&&&
 \TC^1=&  *{\bullet} \ar @{-} @/_.2pc/[rrrr] \ar @{-} @/_1.5pc/[rrrr]\ar@{-} @/^1.1pc/[rrrr]&&&&*{\bullet} \ar@{-}[rr]^(1.1){e_0}_(0,1){v}&\ar@{.}[rrr]&&\ar@{-}[rr]&&*{\circ}\ar@{{-}{.}}[rrru]^(0.6){l_2}\ar@{-}[r]_(0,1){v_0}\ar@{.}[rrr]_(0.5){l_1'}&&&&&&&&&&&&&&&&\\
&&&&&&&&&&&&&&&&&&&&&&\\
}
\end{equation*}
\end{figure}

Now we can state
\begin{proposition}
The tropical forgetful map $\pi^\trop:\oM_{g,n+1}^\trop \to \oM_{g,n}^\trop$
admits $n$ continuous sections $\sigma_i^\trop: \oM_{g,n}^\trop\to :\oM_{g,n+1}^\trop$,
with $\sigma_i^\trop(\TC):=\TC^i$ for every $\TC\in  \oM_{g,n}^\trop$. The diagram
$$
\xymatrix{
\oM_{g,n}^\an\ar[d]_{\sigma_i^\an}\ar[rr]^\Trop && \oM_{g,n}^\trop\ar[d]^{\sigma_i^\trop} \\
\oM_{g,n+1}^\an\ar[rr]^\Trop && \oM_{g,n+1}^\trop
}
$$ is commutative.
\end{proposition}
\begin{proof}
It is clear that $\TC^i\in \oM_{g,n+1}^\trop$ and that the map $\sigma_i^\trop$ is continuous.
We need to prove that $\pi^{\trop}(\TC^i)=\TC$. Indeed removing the last leaf from $\TC^i$ gives a tropical curve $(\TC^i)_*$ which is not stable,  as the vertex $v_0$ has valency 2. Hence
$\pi^{\trop}(\TC^i)=\widehat{(\TC^i)_*}$; as $\widehat{(\TC^i)_*}=\TC$ the first statement is proven. The proof of commutativity is identical to the proof of commutativity of the clutching diagram below. \end{proof}

\subsection{Tropical clutching   maps} In the algebro-geometric setting, if $g=g_1+g_2$ and $n = n_1+n_2$,
always assuming $2g_i-2+n_i>0$, we have the   so-called {\it clutching} maps  $\kappa=\kappa_{g_1,n_1,g_2,n_2}$
$$\begin{aligned}
 \ocM_{g_1,n_1+1}\  \  &  \times \   \  \  \  \  \ocM_{g_2,n_2+1}  & \stackrel{\kappa}{\longrightarrow} &\  \ocM_{g,n}\  \  \  \  \  \\
(C_1;p_1^1,\ldots,p^1_{n_1+1})&\   , \  (C_2;p^2_1,\ldots,p^2_{n_2+1})   & \mapsto & \  (C ;p_1,\ldots,p_{n}).
  \end{aligned}
$$

These are obtained  by gluing     $C_1$   with $C_2$ by identifying $p^1_{n_1+1}=p^2_{n_2+1}$ 
in such a way that in $C$ the intersection of
 $C_1$ with $C_2$ consists of  exactly  one (separating) node.  We also write $\kappa$ for the induced clutching map on coarse moduli schemes $\oM_{g_1,n_1+1} \times \oM_{g_2,n_2+1}  \stackrel{\kappa}{\longrightarrow}  \oM_{g,n}$.

We now construct the analogous maps in the tropical setting,
always keeping the numerical assumptions of the algebraic case.
To define the tropical clutching map,
 $$\begin{aligned}
\kappa^{\trop}:&\:\oM_{g_1,n_1+1}^\trop &\times& \oM_{g_2,n_2+1}^\trop   &  \longrightarrow  &  \oM_{g,n}^\trop  \\
&\:\:\; \;\;\;\TC_1\   &\  , &\; \;\;\TC_2\  \  \  \  \   \  \  & \mapsto &\;  \TC\  \   \     
  \end{aligned}
$$
we attach  the last leg of $\TC_1$ (adjacent to the vertex $v_1$)  to the last leg of $\TC_2$ (adjacent to   $v_2$)  by identifying their infinite points

$$\xymatrix{
*{\ldots\bullet}  \ar@{-}[rr]_(0.1){v_1}&\ar@{.}[rr]_(1){\infty_1}&& {{\scriptstyle {\bullet}}}\ar@{<~>}[r] &{{\scriptstyle{\bullet}}}\ar@{.}[rrr]_(0.03){\infty_2}&&&*{\bullet\ldots} \ar@{-}[ll]^(0.1){v_2}
}$$
to form an edge $e$ of $\TC$ with endpoints $v_1$ and $v_2$: 
$$\xymatrix{
 *{\ldots\bullet}\ar@{-}[rrrr]_(0.1){v_1}_(0.9){v_2}^(0.5){e}
  &&&& *{\bullet\ldots}   
}$$
The length of $e$ is, quite  naturally, set to be equal to $\infty$.
Since the new edge $e$ is a bridge (i.e. a disconnecting edge) of the underlying graph, we have that the genus of $\TC$ is equal to the sum of the genera of $\TC_1$ and $\TC_2$. Hence we have that the image of $\kappa^\trop$ lies in $\oM_{g,n}^\trop$.
Observe that this image     is entirely contained in the locus of extended tropical curves having at least one bridge of infinite length.

\

  \noindent{\it {Proof of the commutativity of the clutching  diagram of Theorem~\ref{Th:tautological}.}}
We begin by reviewing the map $\kappa^\an$ of our diagram, obtained by passing to coarse moduli spaces and then taking analytifications:
$$\xymatrix{
\oM_{g_1,n_1+1}^\an\times \oM_{g_2,n_2+1}^\an \ar[rrr]^{\Trop\times \Trop}\ar[d]_{\kappa^\an} &&&
 \oM_{g_1,n_1+1}^{\trop} \times \oM_{g_2,n_2+1}^{\trop}\ar[d]^{\kappa^\trop} \\
\oM_{g,n}^\an\ar[rrr]^{\Trop} && &
\oM_{g,n}^{\trop}.
}$$
The map $\kappa^\an$ is defined by functoriality of analytification, and can be understood as follows. Fix again an algebraically closed field $K$ with valuation $\val: K \to \RR \sqcup \{\infty\}$, valuation ring $R$ and special point $s$.  Consider a $K$-point  
$[C^1,C^2]\in \oM_{g_1,n_1+1}^\an\times \oM_{g_2,n_2+1}^\an$.
The point is simply  a morphism $\Spec K \to \oM_{g_1,n_1+1}\times \oM_{g_2,n_2+1}$; since the coarse moduli space is proper it  extends to a morphism 
we denote $$\mu_1\times \mu_2:\Spec R\to \oM_{g_1,n_1+1}\times \oM_{g_2,n_2+1}.$$ 

For $i=1,2$ the two projections  provide us with $K$ valued points   $[C^i]\in \oM_{g_i,n_i+1}^\an$ represented by
$$
 \mu_i: \Spec R\to \oM_{g_i,n_i+1}, 
 $$ 
giving two stable pointed curves $C^i\to \Spec R$.

Write $\kappa^\an([C^1,C^2]) = [C]\in \oM_{g,n}^\an$; it is represented by the composition   
 $$
 \kappa\circ(\mu_1\times \mu_2):\Spec R \to \oM_{g,n},
$$
 gluing the two families of curves $C_i\to \Spec R$ along the two sections 
$\sigma_{n_i+1}: \Spec R \to C_i.$

Now, $\Trop\times\Trop  ([C^1,C^2])=\bigr((\WG_1,\ell_1),(\WG_2,\ell_2)\bigl)$
 with $\WG_i = \WG_{C^i_s}$ and $\ell_i$ defined  in Definition~\ref{Def:Trop}. Next
 $$
 \kappa^{\trop}(\Trop\times\Trop ([C^1,C^2]))=(\WG, \widetilde{\ell})
$$
 where, according to our  description above,
 $\WG$ is obtained from $\WG_1$ and $\WG_2$ by merging their respective last legs into one edge, denoted by $e$ (which is thus a bridge of $\WG$).
The definition of $\widetilde{\ell}$ is as follows: 
  \begin{displaymath}
\widetilde{\ell}(\widetilde{e})=\left\{ \begin{array}{ll}
\ell_i(\widetilde{e}) &\text{ if } \widetilde{e}\in E(\WG_i), \quad i=1,2.  \\
+\infty  &\text{ otherwise i.e. if } \widetilde{e}=e.\\
\end{array}\right.
\end{displaymath}
 Consider now $[C]$ and its associated family, 
 $C\to \Spec R$.
 We have a diagram
 $$\xymatrix{
C^1\sqcup C^2\ar[r]^(0.6){\eta}\ar[dr]&C\ar[d]\\
&\Spec R 
}$$
 where  
 the map $\eta$ glues together the last marked points of $C^1$ and $C^2$; let us write $C^*=C^1\sqcup C^2$
 for simplicity.

The special fiber of $C$ is $C_s=\kappa([C^1_s, C^2_s])$ and hence its dual graph
 is equal to $\WG$.

 Let us look at the local geometry   at a node   of $C_s$.
Pick the node corresponding to the new edge $e$, then the generic fiber $C_K$
 has a node specializing to it (the node corresponding to the gluing of the last marked point of $C^1_K$ with the last marked point of $C^2_K$) and hence the local equation at this node
 is $xy=0$. Therefore $\ell_{C}(e)=\infty =\widetilde{\ell}(e)$, as required.
 
 Consider now a node corresponding to an edge $\widetilde{e}\neq e$; without loss of generality this edge $\widetilde{e}$ corresponds to a node of
 $C^1_s$, at which the local equation of $C^1$ is
 $xy=f$ with $f\in R$.  This also serves as a  local equation of
 $C^*$ at the corresponding node, and since it is disjoint from $\sigma_{n +1}$, also  a  local equation of
 $C$. Therefore
 $$
 \ell_{C}(\widetilde{e})=\val(f)=\widetilde{\ell}(\widetilde{e})
 $$
and we are done. \qed
\subsection{Tropical  gluing maps}
In the algebraic setting, for $g>0$ there is a map $$\gamma: \ocM_{g-1,n_1+2} \to \ocM_{g,n}$$ 
obtained by gluing the last two marked points.  We also write $\gamma$ for the induced gluing map on coarse moduli spaces $\gamma: \oM_{g-1,n_1+2} \to \oM_{g,n}.$ 
We  now define the  tropical gluing maps 
  $$\gamma^\trop:\oM_{g-1,n+2}^\trop \longrightarrow \oM_{g,n}^\trop$$
   (always assuming $g>0$).  The procedure is similar to
   the definition of the tropical clutching map;
  $\gamma^\trop$ maps  a tropical curve $\TC$ with $n+2$ legs
  to the tropical curve $\TC'$ obtained by attaching the last two legs of $\TC$, so as to form
  an edge $e'$ of infinite length for $\TC'$. It is clear that $\TC'$ has now only $n$ legs, and its genus is
  gone up by one, as the new edge $e'$ is not a bridge of $\TC'$.

 \noindent{\it {Proof of the commutativity of the gluing  diagram of Theorem~\ref{Th:tautological}.}}
The diagram whose commutativity we must prove is the following.
$$ \xymatrix{
\oM_{g-1,n+2}^\an\ar[rr]^{\Trop}\ar[d]_{{\gamma}^\an} &&
 \oM_{g-1,n+2}^{\trop}\ar[d]^{\gamma^\trop} \\
\oM_{g,n}^\an\ar[rr]^{\Trop} && 
\oM_{g,n}^{\trop}.
}
$$
The proof follows the same pattern 
used to prove the commutativity of the first diagram in the theorem.
Let $[C]\in \oM_{g-1,n+2}^\an$ 
be     represented by the pair 
$$ 
(\val_C:K\longrightarrow \RR\sqcup \{\infty\}, \  \  \mu_C:\Spec R \longrightarrow \oM_{g-1,n+2}).
$$ 
Denote by $\WG$ the dual graph of $C_s$.
Now set 
$\gamma^\an([C])=[C']\in \oM_{g,n}^\an$,   represented by the pair
$$
(\val_C:K\longrightarrow \RR\sqcup \{\infty\}, \  \  \gamma\circ\mu_C:\Spec R \longrightarrow \oM_{g,n}).
$$
The special fiber $C'_s$ of $C'\to \Spec R$ is equal to $\gamma(C_s)$.
It is clear that the graph underlying $\Trop(\gamma^\an([C]))$ and 
the graph underlying $\gamma^\trop(\Trop([C]))$ are isomorphic to the dual graph
of $\gamma(C_s)$, denoted by $\WG'$.
It remains to show the length functions on $E(\WG')$
of $\Trop(\gamma^\an([C]))$ and $\gamma^\trop(\Trop([C]))$
 coincide.

Recall that $\WG'$ is obtained from $\WG$ by adding a new edge, $e'$, joining the endpoints
of the last two legs (and removing these   two legs). The length
of $e'$ in the tropical curve $\gamma^\trop(\Trop([C]))$ is set to be equal to $\infty$,
whereas the length of every other edge $e\in E(\WG')\smallsetminus\{e'\}=E(\WG)$
is $\ell_C(e)$.

Now consider $\Trop(\gamma^\an([C]))=(\WG', \ell_{C'}$); recall that $\ell_{C'}$ depends on the local
geometry of   $C'\to \Spec R$ near the nodes of $C_{s}'$.
Consider the node $\sp'$ corresponding to the new  edge $e'$;
 the generic fiber of  $C'\to \Spec R$ also has   a node specializing to $\sp'$, therefore the local equation
 of $C'$ at $\sp'$ is $xy=0$. Hence 
 $$\ell_{C'}(e')=\val_C(0)=\infty$$ just as in $\gamma^\trop(\Trop([C]))$. Locally at every other node of $C_{s}'$  we have that $C$ and $C'$ are isomorphic,
 hence on the corresponding edges of $\WG'$ we have $\ell_{C'}=\ell_C$ .
 The proof is now complete \qed.

\subsection{Functorial interpretation of the maps} We have defined tropical forgetful, clutching and gluing maps, as well as sections, using the modular meaning of $\oM_{g,n}^\trop$.  Theorem \ref{Th:functor} allows us to interpret these maps in terms of the functorial properties of the maps $\boldp$. 

First note that all the algebraic tautological maps are sub-toroidal: the map $\pi$ is toroidal since it is a family of nodal curves, see \cite[2.6]{Abramovich-Karu}. The section $\sigma_i$ is an isomorphism onto a toroidal substack, and the clutching and gluing maps 
 factor through an \'etale covering of degree one or two followed by a normalization map, by    \cite[Proposition XII.10.11]{Arbarello-Cornalba-Griffiths},
so they are indeed sub-toroidal. By Proposition \ref{Prop:functoriality} we have a commutative diagram
$$\xymatrix{
\oM_{g,n+1}^\an\ar[rr]^\boldp \ar[d]_{\pi^\an} && \oSigma(\ocM_{g,n+1})\ar[d]^{\oSigma(\pi)} \\
\oM_{g,n}^\an\ar[rr]^\boldp  && \oSigma(\ocM_{g,n}) 
}$$
and similarly for the maps $\sigma_i,\gamma,\kappa$. Theorem \ref{Th:functor}  extends this to a commutative diagram
$$\xymatrix{
\oM_{g,n+1}^\an\ar[rr]_\boldp \ar[d]_{\pi^\an}\ar@/^1.5pc/[rrrr]^{\Trop}&& \oSigma(\ocM_{g,n+1})\ar[d]^{\oSigma(\pi)}\ar[rr]_\oPhi^{\sim}&&\oM_{g,n+1}^\trop\ar@{.>}[d] \\
\oM_{g,n}^\an\ar[rr]^\boldp\ar@/_1.5pc/[rrrr]_{\Trop}  && \oSigma(\ocM_{g,n}) \ar[rr]^\oPhi_{\sim}&&\oM_{g,n+1}^\trop.
}$$
Since the two arrows designated by $\oPhi$ are isomorphisms,  there is necessarily a unique arrow $\oPhi\circ\oSigma(\pi)\circ \oPhi^{-1}$ making the diagram commutative;  it therefore must coincide with the map $\pi^\trop$ we defined above. The same holds for the maps  $\sigma_i,\gamma,$ and $\kappa$.

\subsection{Variations on the tropical  gluing and clutching maps}
In the algebro-geometric situation,
the clutching  and gluing maps together cover the entire boundary of  $\ocM_{g,n}$, since the result of desingularizing a node while adding its two branches as marked points
is either the disjoint union of two stable curves with suitable genera $g_1,g_2$, and suitable $n_1+1$ and $n_2+1$ marked points, or one curve of genus $g-1$ with $n+2$ marked points.
The situation is quite different in the tropical setting, indeed we have the following fact.
\begin{lemma}  In $ \oM_{g,n}^\trop$ the union of the image of $\gamma^\trop_{g,n}$
 with the images of all the clutching maps $\kappa^\trop_{g_1,n_1,g_2,n_2}$ is equal to 
 $\oM^{\trop}_{g,n}\smallsetminus M^{\trop}_{g,n}$, i.e. to
 the locus of tropical curves having at least one edge of infinite length.
\end{lemma}
\begin{proof}
We just need to prove that a point of $\TC\in \oM^{\trop}_{g,n}\smallsetminus M^{\trop}_{g,n}$ lies in the image of a clutching or gluing map. Let $e$ 
be an edge of $\TC$ having infinite length, write $v_1$ and $v_2$ for its (possibly equal) endpoints.  Let $\TC'$ be the tropical curve
obtained by removing $e$ and attaching a leg $l_1$ at $v_1$ and a leg $l_2$ at $v_2$.
 If $e$ is not a bridge, $\TC'$ is easily seen to be a stable tropical curve of genus $g-1$ with
 $n+2$ marked ponts (which we order so that $l_1$ and $l_2$ are the last ones); it is clear that the image of $\TC'$ via the gluing map is $\TC$.
 
 If $e$ is a bridge, then $\TC'=\TC_1\sqcup \TC_2$,   with $\TC_i$ containing the vertex $v_i$; it is clear that $\TC_i$ is a    stable tropical curve 
 whose last leg we set equal to $l_i$, for $i=1,2$.
Then $\TC$ is equal to $\kappa^{\trop}(\TC_1,\TC_2)$.
\end{proof}

The locus of smooth algebraic curves in $\ocM_{g,n}$ corresponds 
in the tropical moduli space to
the smallest stratum,
 that is the
 single point 
 ${\bullet_{g,n}}\in \oM^{\trop}_{g,n}$ parametrizing 
the tropical curve whose   graph has a unique vertex of weight $g$  (and no edges).
 Hence  the boundary
  of $\ocM_{g,n}$  corresponds 
the open subset   $\oM^{\trop}_{g,n}\smallsetminus \{{\bullet_{g,n}}\}.$

In this sense, a  tropical counterpart of the fact that the clutching and gluing maps cover 
the boundary of  $\ocM_{g,n}$ should be that some generalized tropical gluing and clutching maps cover $\oM^{\trop}_{g,n}\smallsetminus \{{\bullet_{g,n}}\}.$ 

We shall now define a generalization of the previously defined  gluing and clutching maps having that goal in mind.

We denote $\RR_+=\RR_{>0}\sqcup \{\infty\}$.
For every pair $(x,y)\in \RR_+\times \RR_+$,
we have a new tropical gluing map  $\gamma^{\trop}[x,y]$
$$\gamma^\trop[x,y]:\oM_{g-1,n+2}^\trop  \to \oM_{g,n}^\trop$$
constructed as follows. Denote  by $l_{n+1}$ and $l_{n+2}$ the last two legs of $\TC'\in\oM_{g-1,n+2}^\trop $, and  by $v_{n+1}$ and $v_{n+2}$ the vertex they are adjacent to.
As in the definition of $\gamma$ in the previous section, we send $\TC'$
  to a curve in $\TC\in \oM_{g,n}^\trop$ by merging $l_{n+1}$ and $l_{n+2}$ 
into one edge $e$ of $\Gamma$. The difference is that now the new edge will have length
equal to $\ell(e)=x+y$. This is obtained by fixing on $l_{n+1}$ 
a point $p_{n+1}$ of distance $x$ from $v_{n+1}$, and   
``clipping off" the remaining infinite line; similarly, we fix a point $p_{n+2}$ of distance $y$ from $v_{n+2}$,
on $l_{n+2}$ 
and disregard the rest of the leg; then we glue  $p_{n+1}$ to $p_{n+2}$ obtaining an edge between $v_{n+1}$ and $v_{n+2}$  of length $x+y$.
It is clear that $\gamma^{\trop}[x,y]$ is  continuous.
Observe that $\gamma^{\trop}[x,y]$ depends only on $x+y$ (and hence it is symmetric) and that
  the ``more natural" gluing map $\gamma^\trop$ defined before is obtained as 
$$\gamma^\trop=\gamma^\trop[\infty, x]=\gamma^\trop[\infty, \infty].$$
Summarizing, we have defined a continuous family of maps
$$
\gamma^\trop[\,,\,]:\oM_{g-1,n+2}^\trop  \times \RR_+\times \RR_+\longrightarrow \oM_{g,n}^\trop.
$$
In a completely analogous way we define the generalized clutching maps
$\kappa^\trop[\,,\,]=\kappa^\trop_{g_1,n_1,g_2,n_2}[\,,\,]$:

$$
\kappa^\trop[\,,\,]:\oM_{g_1,n_1+1}^\trop\times  \oM_{g_2,n_2+1}^\trop \times \RR_+\times \RR_+\longrightarrow\oM_{g,n}^\trop.
$$
As before, $\kappa^\trop[x,y]=\kappa^\trop[x,y]$ and the original clutching map is recovered as 
$$\kappa^{\trop}=\kappa^{\trop}[x, \infty]=\kappa^{\trop}[\infty, \infty].
$$
\begin{remark}
It is clear that the union of the image of $\gamma_{{g-1,n+2}}^\trop[\,,\,]$
with the images of the maps $\kappa^\trop_{g_1,n_1,g_2,n_2}[\,,\,]$
is equal to $\oM^{\trop}_{g,n}\smallsetminus \{{\bullet_{g,n}}\}.$
\end{remark}
 
\begin{remark}
These generalized clutching and gluing maps naturally lift to Berkovich analytic spaces in the following way, as suggested by a referee.  We describe the construction for the clutching map; the gluing map may be handled similarly.  Factor the algebraic clutching map $\kappa$ as
\[
\xymatrix{ && \ocM_{g,n+1}\ar[d]^\pi\\
\ocM_{g_1,n_1+1}\times \ocM_{g_2,n_2+1} \ar^{\kappa'}[urr]  \ar[rr]^(.7)\kappa&&\ocM_{g,n} 
}
\]
by setting  $\kappa'([C_1],[C_2]) =[C_1\cup \PP^1\cup C_2]$ where $\PP^1$ is a projective line attached at 0 to the $(n_1+1)$st point of $C_1$, attached at $\infty$ to  the $(n_2+1)$st point of $C_2$, and having the new $(n+1)$st  marked point at $1\in \PP^1$.

Let $U \subset \oM_{g, n+1}^\an$ be the open subset parametrizing curves whose reduction is in the image of $\kappa'$.  Then $U$ is the preimage of an open subset of $\oSigma(\oM_{g,n+1})$ which is naturally identified with $\oM_{g_1,n_1+1}^\trop\times  \oM_{g_2,n_2+1}^\trop \times \RR_+\times \RR_+ $.  Composing with the forgetful map gives $U \rightarrow \oM_{g,n}^\an$, and the induced map on skeletons is precisely $\kappa^\trop[ \, , \, ]$.


\end{remark}

\providecommand{\bysame}{\leavevmode\hbox to3em{\hrulefill}\thinspace}
\providecommand{\MR}{\relax\ifhmode\unskip\space\fi MR }
\providecommand{\MRhref}[2]{%
  \href{http://www.ams.org/mathscinet-getitem?mr=#1}{#2}
}
\providecommand{\href}[2]{#2}

\end{document}